\documentclass{article}
\usepackage{amsmath,amssymb,amscd}
\usepackage{amsthm}
\newtheorem{thm}{\indent \sc Theorem}[section]
\newtheorem{prop}[thm]{\indent \sc Proposition}
\newtheorem{cor}[thm]{\indent \sc Corollary}
\newtheorem{defi}[thm]{\indent \sc Definition}
\newtheorem{lem}[thm]{\indent \sc Lemma}
\newtheorem{rem}[thm]{\indent \sc Remark}

\newcommand{\spec}{\operatorname{Spec}}
\newcommand{\et}{\mathrm{\acute{e}t}}
\newcommand{\Nis}{\mathrm{Nis}}
\newcommand{\Zar}{\mathrm{Zar}}

\newcommand{\HO}{\operatorname{H}}
\newcommand{\HB}{\operatorname{H}_{\mathrm{B}}}

\usepackage[all]{xy}

\begin{document}

\author{Makoto Sakagaito\footnote{
\textit{Present affiliation}: 
Indian Institute of Science Education and Research, Bhopal\newline
\textit{E-mail address}: makoto@iiserb.ac.in, sakagaito43@gmail.com}
}
\title{On a generalized Brauer group 
in mixed characteristic cases}
\date{}
\maketitle
\begin{center}
Indian Institute of Science Education and Research, Mohali\\
\end{center}
\begin{abstract}
We define a generalization of the Brauer group
$\HB^{n}(X)$ for an equi-dimensional scheme $X$ 
and $n>0$.
In the case where $X$ is the spectrum of a local ring of a 
smooth algebra over a discrete valuation ring,
$\HB^{n}(X)$ 
agrees with the \'{e}tale motivic cohomology 
$\HO^{n+1}_{\mathrm{\acute{e}t}}\left(X, \mathbb{Z}(n-1)\right)$. 
We prove
(a part of) the Gersten-type conjecture for
the generalized Brauer group for
a local ring of a 
smooth algebra over a mixed characteristic 
discrete valuation ring
and
an isomorphism
\begin{math}
\HB^{n}\left(
R
\right)
\simeq
\HB^{n}\left(
k
\right)
\end{math}
for a henselian local ring $R$ of a smooth algebra over a
mixed characteristic discrete valuation ring and the
residue field $k$ of $R$.
As an application, we show 
local-global principles
for Galois cohomology groups 
over function fields of smooth curves 
over 
a mixed characteristic
excellent henselian discrete valuation ring.

\end{abstract}
%
%
\section{Introdution}

Let $A$ be a Dedekind ring or field,
$X$ a smooth scheme over $\spec(A)$,
$\mathbb{Z}(n)_{\et}$ the Bloch's cycle 
complex for \'{e}tale topology
and 
\begin{math}
\mathbb{Z}/m(n)_{\et}
=\mathbb{Z}(n)_{\et}\otimes\mathbb{Z}/m\mathbb{Z}
\end{math}
for a positive integer $m$.
Let
$D^{b}(X_{\et}, \mathbb{Z}/m\mathbb{Z})$
be the derived category of
bounded complexes of \'{e}tale
$\mathbb{Z}/m\mathbb{Z}$-sheaves on $X$.

Then
\begin{description}
\item{(i)}
If $l\in \mathbb{N}$ is invertible in $A$,
there is a quasi-isomorphism
\begin{equation*}
\mathbb{Z}/l(n)_{\et}
\xrightarrow{\sim}
\mu^{\otimes n}_{l}[0]
\end{equation*}
in
$D^{b}(X_{\et}, \mathbb{Z}/l\mathbb{Z})$
by (\cite[p.774, Theorem 1.2.4]{Ge}, \cite{V2}).
Here $\mu_{l}$ is the sheaf of
$l$-th roots of unity.

\item{(ii)}
If $A$ is a field of characteristic $p>0$, there is a quasi-isomorphism
\begin{equation*}
\mathbb{Z}/p^{r}(n)_{\et}
\xrightarrow{\sim}
\operatorname{W}_{r}\Omega^{n}_{X, \log}[-n]
\end{equation*}
in
$D^{b}(X_{\et}, \mathbb{Z}/p^{r}\mathbb{Z})$
for any positive integer $r$
by \cite[p.787, \S 5, (12)]{Ge}.
Here $\operatorname{W}_{r}\Omega^{n}_{X, \log}$ is 
the logarithmic de Rham-Witt sheaf.
\end{description}

\vspace{2.0mm}

Assume that a positive integer $m$ equals $l$ in (i) or
$p^{r}$ in (ii). 
Let $R$ be a local ring of a smooth algebra over $A$, 
$k(R)$ its fraction field and 
$\kappa(\mathfrak{p})$ the residue field of 
$\mathfrak{p}\in\spec R$. 
Let $n$ be a positive integer. Suppose that
$n\geq N$
in the case where
$A$ is a field of characteristic $p>0$
and
$[k(R): k(R)^{p}]=N$.

Then the sequence of \'{e}tale hyper-cohomology groups
\begin{align}\label{KC}
0
\to
\HO^{n+1}_{\et}
\left(
R,
\mathbb{Z}/m(n)
\right)  
\to&
\HO^{n+1}_{\et}
\left(
k(R),
\mathbb{Z}/m(n)
\right)  \\
\to& 
\displaystyle
\bigoplus_{
\substack
{
\mathfrak{p}\in \spec R  \nonumber \\
\operatorname{ht}(\mathfrak{p})
=
1
}
}
\HO^{n}_{\et}
\left(
\kappa(\mathfrak{p}),
\mathbb{Z}/m(n-1)
\right)  
\end{align}
is exact by  
(\cite[p.774, Theorem 1.2.(2, 4, 5)]{Ge}, \cite{V2}) and 
\cite[p.608, Theorem 5.2]{Sh}.

\vspace{2.0mm}

Since 
\begin{math}
\HO_{\Zar}^{i}
\left(
R, \mathbb{Z}(n)
\right)
=
0
\end{math}
for $i>n$
by \cite[p.779, Theorem 3.2 b)]{Ge} and
\cite[p.786, Corollary 4.4]{Ge},
we have
\begin{equation*}
\HO_{\et}^{n+1}
\left(
R, \mathbb{Z}(n)
\right)
=
\HO_{\et}^{i}
\left(
R, \mathbb{Q}(n)
\right)
=
0
\end{equation*}
for $i>n$ by 
\cite[p.774, Theorem 1.2.2]{Ge}
and 
\cite[p.781, Proposition 3.6]{Ge}.
Hence the \'{e}tale motivic cohomology
\begin{math}
\HO_{\et}^{n+2}
\left(
R, \mathbb{Z}(n)
\right)
\end{math}
is a torsion group and
\begin{equation*}
\HO^{n+1}_{\et}
\left(
R,
\mathbb{Z}/m(n)
\right)
=
\operatorname{Ker}
\left(
\HO^{n+2}_{\et}
\left(
R,
\mathbb{Z}(n)
\right)
\xrightarrow{\times m}
\HO^{n+2}_{\et}
\left(
R,
\mathbb{Z}(n)
\right)
\right)
\end{equation*}
for any positive integer $m$.
Therefore we can regard the sequence (\ref{KC}) as 
(a part of) the Gersten type resolution for the \'{e}tale motivic cohomology.

\vspace{2.0mm}

One of the objectives of this paper is to  
prove that 
an improved version of the sequence (\ref{KC}) 
is exact
in the case where
$A$
is a discrete valuation ring of mixed-characteristic
$(0, p)$ and $m=p^{r}$
as follows:

\begin{thm}\upshape(Proposition \ref{CaseGBEt}
and Theorem \ref{GBmixex})
\label{MTGC}
Let $A$ be a discrete valuation ring of 
mixed-characteristic
$(0, p)$,
$R$ a local ring of a smooth algebra over $A$,
$k(R)$ the fraction field of $R$ and 
$\kappa(\mathfrak{p})$ the residue field of
$\mathfrak{p}\in \spec R$.

Then the sequence
\begin{align}\label{GBeq}
0
\to
\HO^{n+1}_{\et}
\left(
R,
\mathbb{Z}/p^{r}(n)
\right)  
\to&
\HO^{n+1}_{\et}
\left(
k(R),
\mathbb{Z}/p^{r}(n)
\right)^{\prime}  \\
\to& 
\displaystyle
\bigoplus_{
\substack
{
\mathfrak{p}\in \spec R  \nonumber \\
\operatorname{ht}(\mathfrak{p})
=
1
}
}
\HO^{n}_{\et}
\left(
\kappa(\mathfrak{p}),
\mathbb{Z}/p^{r}(n-1)
\right)  
\end{align}
is exact
for
integers $n\geq 0$
and
$r> 0$
where     
\begin{align*}
&\HO^{n+1}_{\et}
\left(
k(R),
\mathbb{Z}/p^{r}(n)
\right)^{\prime}     \\
=
&\operatorname{Ker}
\left(
\HO^{n+1}_{\et}
\left(
k(R),
\mathbb{Z}/p^{r}(n)
\right)
\to
\displaystyle
\prod_{
\substack{
\mathfrak{p}\in \spec R  \\
\operatorname{ht}(\mathfrak{p})=1
}
}
\HO^{n+1}_{\et}
\left(
k(
R_{\bar{\mathfrak{p}}}
),
\mathbb{Z}/p^{r}(n)
\right)
\right)
\end{align*}
and
\begin{math}
R_{\bar{\mathfrak{p}}}
\end{math}
is the strictly henselization of
$R_{\mathfrak{p}}$.
\end{thm}

This paper is organized as follows.
In \S\ref{EMC} we prove Proposition \ref{tru2}
which is an improved version of 
the purity theorem of
the Bloch's cycle complex 
for \'{e}tale topology
(\cite[p.774, Theorem 1.2.1]{Ge})
by improving the proof of
\cite[p.774, Theorem 1.2.1]{Ge}
and prove that the sequence (\ref{GBeq}) is exact 
in the case where
$R$ is a discrete valuation ring.

In \S\ref{defgb} we define a generalization of the Brauer group. 
We show the reason why we can regard 
it as a generalization of the Brauer group 
(Proposition \ref{reason}) and
a relation between it and \'{e}tale motivic cohomology
(Proposition \ref{CaseGBEt}).

In \S\ref{Puri}
we prove 
(a part of) the Gersten-type conjecture 
for the generalized Brauer group
for a local ring of a smooth algebra 
over a mixed characteristic discrete
valuation ring
(Theorem \ref{GBmixex}).
In \S\ref{Rigid}
we prove the rigidity theorem 
for 
\'{e}tale motivic cohomology
for
a henselian local ring
of a smooth algebra over
a mixed-characteristic discrete valuation ring .

Namely we prove the following theorem:
\begin{thm}\upshape(
Theorem \ref{henmix}) \label{Ihenmix}
Let $R$ be a henselian local ring of a smooth algebra over
a 
mixed-characteristic
discrete valuation ring $A$ and
$k$ the residue field of $R$. 

Then we have an isomorphism
\begin{equation*}
\HO^{n+1}_{\et}
\left(
R,
\mathbb{Z}/m(n)
\right)
\xrightarrow{\sim}
\HO^{n+1}_{\et}
\left(
k,
\mathbb{Z}/m(n)
\right)
\end{equation*}
for any positive integer $m$.
\end{thm}
Theorem \ref{Ihenmix} is proved in the case where $m$ is invertible
in $A$ by (\cite[p.774, Theorem 1.2.3]{Ge}, \cite{V2}).

\vspace{2.0mm}

Finally we prove the following local-global principle in \S\ref{App}
by applying Theorem \ref{MTGC} and Theorem \ref{Ihenmix}.
%
\begin{thm}\upshape(Theorem \ref{LGG})
Let $A$ be an excellent henselian discrete valuation ring
of mixed characteristic
$(0, p)$ and 
$\pi$ a prime element of $A$.
Let $\mathfrak{X}$ be 
a connected proper smooth curve over $\spec A$,
$K$ the fraction field of $\mathfrak{X}$
and
$K_{(\eta)}$ the fraction field of
the henselization of
$\mathcal{O}_{\mathfrak{X}, \eta}$.

Then the local-global map
\begin{equation}\label{LGGI}
\HO^{n+1}_{\et}
\left(
K,
\mu^{\otimes n}_{m}
\right)
\to
\displaystyle
\prod_{
\substack{
\eta\in \mathfrak{X}^{(1)} 
}
}
\HO^{n+1}_{\et}
\left(
K_{(\eta)},
\mu^{\otimes n}_{m}
\right)
\end{equation}
is injective
for
integers $n\geq 0$
and
$m=p^{r}$
where $\mathfrak{X}^{(1)}$ is
the set of points of codimension $1$.
\end{thm}
Suppose that $\mathfrak{X}$ is a regular scheme which is
flat of finite type over 
an excellent henselian discrete valuation 
ring $A$ and
whose fiber over the closed point of $\spec A$
is reduced normal crossing divisors on $\mathfrak{X}$.
Then the local-global map (\ref{LGGI}) is
injective 
for 
$n\geq 0$ 
and
$m$ which is prime to $\operatorname{char}(A)$
(cf. \cite[Theorem 3.3.6]{H-H-K}, \cite[Theorem 1.2]{Hu}
and \cite[Theorem 1.2]{SakG}).
\subsection*{Acknowledgements}
The first version of this paper (\cite{Sak}) was written in 2013-2016, 
while
the author was staying at Harish-Chandra Research Institute.
He thanks Harish-Chandra Research Institute for the supports
during his study and writing.
\subsection*{Notations}
Throughout this paper, $p$ will be a fixed prime, unless otherwise
stated.
For a scheme $X$, $X_{\et}$, $X_{Nis}$ and $X_{\mathrm{Zar}}$ denote the category of \'{e}tale schemes over $X$ equipped with the \'{e}tale, Nisnevich 
and
Zariski topology, respectively. 
For $t\in \{\et, \Nis, \operatorname{Zar}\}$, $\mathbb{S}_{X_{t}}$ denotes the category of sheaves on $X_{t}$.
$X^{(i)}$ denotes the set of points of codimension $i$ and $X_{(i)}$ denotes the set of
points of dimension $i$. 
$k(X)$ denotes the ring of rational functions on $X$ and
$\kappa(x)$ denotes the residue field of 
$x\in X$.
For scheme over $\mathbb{F}_{p}$,
\begin{math}
\Omega_{X}^{q}
=
\Omega_{X/\mathbf{Z}}^{q}
\end{math}
denotes
the exterior algebra over 
$\mathcal{O}_{X}$
of the sheaf 
$\Omega^{1}_{X/\mathbf{Z}}$
of absolute differentials on $X$
and
$\Omega_{X, \log}^{q}$ 
the part of 
$\Omega^{q}_{X}$
generated \'{e}tale locally by local sections of the
forms
\begin{equation*}
\frac{dx_{1}}{x_{1}}
\wedge
\cdots
\wedge
\frac{dx_{q}}{x_{q}}.
\end{equation*}
\section{\'{E}tale motivic cohomology}\label{EMC}
Let $D_{i}=\mathbb{Z}[t_{0}, \cdots, t_{i}]/(\displaystyle\sum_{j=0}^{i}t_{j}-1)$, and $\Delta^{i}=\spec D_{i}$ be the 
algebraic $i$-simplex. For an equi-dimensional scheme $X$, let $z^{n}(X, i)$ be the free abelian group on closed
integral subschemes of codimension $n$ of $X\times \Delta^{i}$, which intersect all faces property. Intersecting with 
faces defines the structure of a simplicial abelian group, and hence gives a (homological) complex 
$z^{n}(X, *)$.

\vspace{2.0mm}

The complex of sheaves $\mathbb{Z}(n)_{t}$ on the site $X_{t}$, where $t\in \{\et, \mathrm{Zar}\}$, is defined as the cohomological 
complex with
$z^{n}(-, 2n-i)$ in degree $i$. 
For an abelian group $A$ we define $A(n)$ to be $\mathbb{Z}(n)\otimes A$.

\vspace{2.0mm}

First we show an
improved version of 
the purity theorem of
the Bloch's cycle complex 
for \'{e}tale topology
(\cite[p.774, Theorem
1.2.1]{Ge}).

\begin{prop}\upshape
\label{tru2}
Let $A$ be a regular local ring 
with $\operatorname{dim}(A)\leq 1$, $\mathfrak{X}$ a scheme
which is essentially of finite type over 
$\spec A$ and $i:Y\to \mathfrak{X}$ a closed
subscheme of codimension $c$ with open complement $j:X\to \mathfrak{X}$.

Suppose that $X$ is essentially smooth over a regular ring of dimension at most one. 
Then we have a quasi-isomorphism
\begin{equation}\label{puri}
\tau_{\leq n+2}\Bigl(
\mathbb{Z}(n-c)_{\et}[-2c]
\Bigr)
\xrightarrow{\sim} 
\tau_{\leq n+2}\mathbf{R}i^{!}\mathbb{Z}(n)_{\et}
\end{equation}
and a quasi-isomorphism
\begin{equation}
\label{quasiil}
\tau_{\leq n+1}\Bigl(
\mathbb{Z}/m(n-c)_{\et}[-2c]
\Bigr)
\xrightarrow{\sim}\tau_{\leq n+1}\mathbf{R}i^{!}\mathbb{Z}/m(n)_{\et}
\end{equation}
for any positive integer $m$.

\end{prop}
\begin{proof}(cf. The proof of \cite[Theorem 1.2.1]{Ge})
Let 
$X_{\et}\xrightarrow{\epsilon}X_{\Zar}$ 
be the canonical map of sites.
Then we have a quasi-isomorphism
\begin{equation}
\label{Z-ej}
\tau_{\leq n+1}\epsilon^{*}\mathbf{R}j_{*}
\mathbb{Z}(n)_{\mathrm{Zar}}
\xrightarrow{\sim}
\tau_{\leq n+1}\epsilon^{*}\mathbf{R}j_{*}\mathbf{R}\epsilon_{*}\mathbb{Z}(n)_{\et}
\xrightarrow{\sim}
\tau_{\leq n+1}\mathbf{R}j_{*}\mathbb{Z}(n)_{\et}
\end{equation}
by (\cite[p.774, Theorem 1.2.2]{Ge}, \cite{V2}) (cf. \cite[p.787]{Ge}). Since
\begin{equation*}
\mathbf{R}^{n+2}j_{*}\left(
\tau_{\leq n+1}\mathbf{R}\epsilon_{*}\mathbb{Z}(n)_{\et}
\right)
\to 
\mathbf{R}^{n+2}j_{*}\left(
\mathbf{R}\epsilon_{*}\mathbb{Z}(n)_{\et}
\right)
\end{equation*}
is injective by a distinguished triangle
\begin{equation*}
\cdots\to
\tau_{\leq n+1}\mathbf{R}\epsilon_{*}\mathbb{Z}(n)_{\et}
\to
\mathbf{R}\epsilon_{*}\mathbb{Z}(n)_{\et}
\to
\tau_{\geq n+2}\mathbf{R}\epsilon_{*}\mathbb{Z}(n)_{\et}
\to
\cdots
\end{equation*}
and the equation
\begin{equation*}
\mathbf{R}j^{n+1}
(
\tau_{\geq n+2}\mathbf{R}\epsilon_{*}
\mathbb{Z}(n)_{\et}
)
=
0
,    
\end{equation*}
the composite map 
\begin{align*}
&\epsilon^{*}\mathbf{R}^{n+2}j_{*}
\mathbb{Z}(n)_{\mathrm{Zar}}
\xrightarrow{\sim}
\epsilon^{*}\mathbf{R}^{n+2}j_{*}\left(
\tau_{\leq n+1}\mathbf{R}\epsilon_{*}\mathbb{Z}(n)_{\et}
\right) \\
\hookrightarrow
&\epsilon^{*}\mathbf{R}^{n+2}j_{*}\mathbf{R}\epsilon_{*}\mathbb{Z}(n)_{\et}
\xrightarrow{\sim}
\mathbf{R}^{n+2}j_{*}\mathbb{Z}(n)_{\et}
\end{align*}
is injective
by (\cite[p.774, Theorem 1.2.2]{Ge}, \cite{V2}).
Moreover we have the map of 
distinguished triangles
\begin{equation*}
\begin{CD}
\epsilon^{*}\mathbb{Z}(n-c)_{\mathrm{Zar}}[-2c]
@>>>
\epsilon^{*}i^{*}\mathbb{Z}(n)_{\mathrm{Zar}}
@>>>
\epsilon^{*}i^{*}\mathbf{R}j_{*}\mathbb{Z}(n)_{\mathrm{Zar}} \\
@VVV @| @VVV \\
\mathbf{R}i^{!}\mathbb{Z}(n)_{\et}
@>>> 
i^{*}\mathbb{Z}(n)_{\et} 
@>>> 
i^{*}\mathbf{R}j_{*}\mathbb{Z}(n)_{\et}.
\end{CD}
\end{equation*}
Hence we have the quasi-isomorphism (\ref{puri}) 
by the five lemma. Moreover the quasi-isomorphism 
(\ref{quasiil}) follows from the quasi-isomorphism (\ref{puri}).
Therefore the statement follows.

\end{proof}
\begin{rem}\upshape
If $A$ is a Dedekind ring  and $i$ is the inclusion of one of the
closed fibers, we have the quasi-isomorphisms (\ref{puri}) and (\ref{quasiil})
as \cite[p.774, Theorem 1.2.1]{Ge}.
\end{rem}
\begin{lem}\upshape
\label{tru}
Let $F: \mathcal{A}\to \mathcal{B}$ be a left exact functor between abelian categories and
let $\mathcal{A}$ be a Grothendieck category. Then
\begin{equation*}
\mathbf{R}^{n+1}F(\tau_{\leq n}A^{\bullet})
=\operatorname{Ker}\Bigl(\mathbf{R}^{n+1}F(A^{\bullet})\to F(\mathcal{H}^{n+1}(A^{\bullet}))
\Bigr)
\end{equation*}
for any complex $A^{\bullet}$.
\end{lem}
\begin{proof}
If $B^{\bullet}$ is bounded below, there is a convergent spectral sequence for 
the hyper-cohomology with 
\begin{equation}
\label{sshc}
\mathbb{R}^{p}F(\mathcal{H}^{q}(B^{\bullet}))\Rightarrow \mathbb{R}^{p+q}F(B^{\bullet})
\end{equation}
and 
\begin{equation*}
\mathbf{R}^{n}F(B^{\bullet})=\mathbb{R}^{n}F(B^{\bullet}). 
\end{equation*}
So 
\begin{align*}
\mathbf{R}^{n}F(\tau_{\geq n+1}A^{\bullet})&=\mathbb{R}^{n}F(\tau_{\geq n+1}A^{\bullet})=0,\\ 
\mathbf{R}^{n+1}F(\tau_{\geq n+1}A^{\bullet})&=\mathbb{R}^{n+1}F(\tau_{\geq n+1}A^{\bullet})\xrightarrow{\sim}
F(\mathcal{H}^{n+1}(A^{\bullet}))
\end{align*}
for any complex $A^{\bullet}$. By a distinguished triangle
\begin{equation*}
\cdots\to\tau_{\leq n}A^{\bullet}\to A^{\bullet}\to \tau_{\geq n+1}A^{\bullet}\to\cdots,
\end{equation*}
we have a distinguished triangle
\begin{equation*}
\cdots\to\mathbf{R}F\left(\tau_{\leq n}A^{\bullet}\right)\to 
\mathbf{R}F\left(A^{\bullet}\right)\to 
\mathbf{R}F\left(\tau_{\geq n+1}A^{\bullet}\right)\to \cdots.
\end{equation*}
Therefore,
\begin{equation*}
\mathbf{R}^{n+1}F(\tau_{\leq n}A^{\bullet})
=\operatorname{Ker}\Bigl(\mathbf{R}^{n+1}F(A^{\bullet})\to F(\mathcal{H}^{n+1}(A^{\bullet}))
\Bigr).
\end{equation*}
\end{proof}
\begin{rem}\upshape
Let $A^{\bullet}$ be a bounded below complex.
Since the edge maps of spectral sequence are natural maps, the morphism
\begin{equation*}
\mathbf{R}^{n+1}F(A^{\bullet})\to \mathbf{R}^{n+1}F(\tau_{\geq n+1}A^{\bullet})
\end{equation*}
corresponds to the edge map of 
\begin{math}
\mathbb{R}^{p}F(\mathcal{H}^{q}(A^{\bullet}))\Rightarrow \mathbb{R}^{p+q}F(A^{\bullet}).
\end{math}
\end{rem}
\begin{prop}\upshape
\label{supp}
Let $A$ be a discrete valuation ring with 
the fraction field $K$ and
the residue field $k$. 
Let 
$j: \spec K\to \spec A$
be the generic point
and
$K_{\bar{\mathfrak{m}}}$
the maximal unramified extension of $K$.

Then 
\begin{align*}
\operatorname{H}^{n+1}_{\et}
\left(
\spec A, 
\tau_{\leq n}\mathbf{R}j_{*}\mathbb{Z}/m(n)
\right) 
=
\HO^{n+1}_{\et}\left(
K,
\mathbb{Z}/m(n)
\right)^{\prime}
\\
\end{align*}
for integers $n\geq 0$ and $m>0$ where
\begin{align*}
&\HO^{n+1}_{\et}
\left(
K,
\mathbb{Z}/m(n)
\right)^{\prime}
\\
=&\operatorname{Ker}\Bigl(
\operatorname{H}^{n+1}_{\et}
\left(
K, \mathbb{Z}/m(n)
\right)
\to
\operatorname{H}^{n+1}_{\et}
\left(
K_{\bar{\mathfrak{m}}}, 
\mathbb{Z}/m(n)
\right)
\Bigr).
\end{align*}
Moreover the sequence
\begin{align}\label{dim1ex}
0
\to
\HO_{\et}^{n+1}
\left(
A,
\mathbb{Z}/m(n)
\right)
\to
\HO_{\et}^{n+1}
\left(
K,
\mathbb{Z}/m(n)
\right)^{\prime}
\to
\HO^{n}_{\et}
\left(
k,
\mathbb{Z}/m(n-1)
\right)
\end{align}
is exact.
\end{prop}
\begin{proof}
For each $x\in \spec A$, we choose a geometric point 
$u_{x}: \bar{x}\to \spec A$. 

Then 
\begin{equation*}
F\to 
\displaystyle
\prod_{x\in \spec A} (u_{x})_{*}(u_{x})^{*}F 
\end{equation*}
is injective for a sheaf $F$ on 
$(\spec A)_{\et}$ (\cite[p.90, III, Remark 1.20 (c)]{M}). Hence
\begin{align*}
&\operatorname{H}^{n+1}_{\et}
\left(
\spec A, 
\tau_{\leq n}\mathbf{R}j_{*}\mathbb{Z}/m(n)
\right) \\
=&
\operatorname{Ker}\Bigl(
\operatorname{H}^{n+1}_{\et}
\left(
K, \mathbb{Z}/m(n)
\right)
\to 
\Gamma
\left(
\spec A,
\mathbf{R}^{n+1}j_{*}\mathbb{Z}/m(n)
\right)
\Bigr)
\\
=&
\operatorname{Ker}\Bigl(
\operatorname{H}^{n+1}_{\et}
\left(
K, \mathbb{Z}/m(n)
\right)
\to 
\prod_{x\in \spec A}
\left(
\mathbf{R}^{n+1}j_{*}\mathbb{Z}/m(n)
\right)_{\bar{x}}
\Bigr)\\
=&\operatorname{Ker}\Bigl(
\operatorname{H}^{n+1}_{\et}
\left(
K, \mathbb{Z}/m(n)
\right)
\to
\operatorname{H}^{n+1}_{\et}
\left(
K_{\bar{\mathfrak{m}}}, 
\mathbb{Z}/m(n)
\right)
\Bigr)
\end{align*}
by Lemma \ref{tru}.

Let 
$i: \spec k\to \spec A$
the closed immersion.
Since  we have a distinguished triangle
\begin{align*}
\cdots
\to
\mathbb{Z}/m(n)_{\et}
\to
\tau_{\leq n}\mathbf{R}j_{*}\mathbb{Z}/m(n)_{\et}
\to
i_{*}\mathbb{Z}/m(n-1)_{\et}
\to
\cdots
\end{align*}
by Proposition \ref{tru2}
and \cite[p.786, Corollary 4.4]{Ge}, the sequence 
\begin{align*}
\HO^{n}_{\et}
\left(
K,
\mathbb{Z}/m(n)
\right)
\to
\HO^{n-1}_{\et}
\left(
k,
\mathbb{Z}/m(n-1)
\right)
\to
\HO^{n+1}_{\et}
\left(
A,
\mathbb{Z}/m(n)
\right)\\
\to
\HO^{n+1}_{\et}
\left(
\spec A,
\tau_{\leq n}\mathbf{R}j_{*}\mathbb{Z}/m(n)
\right)
\to
\HO^{n}_{\et}
\left(
k,
\mathbb{Z}/m(n-1)
\right)
\end{align*}
is exact.

By \cite[p.774, Theorem 1.2.2]{Ge} 
and \cite[Lemma 3.2]{Ge-L},
the homomorphism
\begin{equation}\label{sym}
\HO^{n}_{\et}
\left(
K,
\mathbb{Z}/m(n)
\right)
\to
\HO^{n-1}_{\et}
\left(
k,
\mathbb{Z}/m(n-1)
\right)
\end{equation}
agrees with the homomorphism
\begin{equation*}
\operatorname{K}^{M}_{n}(K)/m
\to
\operatorname{K}^{M}_{n-1}(
k)/m
\end{equation*}
which is induced by the symbol map of the Milnor $K$-group.

Hence the homomorphism (\ref{sym}) is surjective and 
the homomorphism 
\begin{equation*}
\HO^{n+1}_{\et}
\left(
A,
\mathbb{Z}/m(n)
\right)\\
\to
\HO^{n+1}_{\et}
\left(
\spec A,
\tau_{\leq n}\mathbf{R}j_{*}\mathbb{Z}/m(n)
\right)
\end{equation*}
is injective.
Therefore the statement follows.

\end{proof}
\begin{prop}\upshape
Let $A$ be a henselian discrete valuation ring with the fraction field
 $K$ and the residue field $k$. 
Let $j: \spec K\to \spec A$ be the generic point.
Suppose that $\operatorname{char}(k)=p>0$. Then
\begin{align*}
&\operatorname{H}^{n+1}_{\et}
\left(
\spec A, 
\tau_{\leq n}\left(
\mathbf{R}j_{*}\mathbb{Z}/p^{r}(n)
\right)
\right)\\
=&
\operatorname{H}^{1}
\left(
\operatorname{Gal}(k_{s}/k), 
\operatorname{K}^{M}_{n}(K_{\bar{\mathfrak{m}}})/p^{r}
\right)
\end{align*}
for integers $n\geq 0$ and $r>0$. Here $k_{s}$ is the separable closure
of $k$, $K_{\bar{\mathfrak{m}}}$ is the maximal unramified 
extension of $K$
and 
$\operatorname{K}^{M}_{n}(K_{\bar{\mathfrak{m}}})$
is the $n$-th Milnor $K$-group of 
$K_{\bar{\mathfrak{m}}}$. 
\end{prop}
\begin{proof}\upshape
Let 
\begin{math}
E^{l, m}_{2}\Rightarrow E^{l+m}
\end{math}
be a spectral sequence. Suppose that $E^{l, m}_{2}=0$ for $l<0$ or $m<0$.
Then we have a filtration
\begin{equation*}
0
\subset
F^{l+m}_{l+m}
\subset
\cdots
\subset
F^{p+q}_{1}
\subset
F^{p+q}_{0}
=
E^{p+q}
\end{equation*}
such that
\begin{equation*}
F_{l}^{l+m}/F_{l+1}^{l+m}
\simeq
E^{l, m}_{\infty}
\end{equation*}
and we can define the morphism
\begin{equation}\label{spe}
\operatorname{Ker}\left(E^{n}\to E^{0, n}_{2}\right)\simeq F^{n}_{1}
\twoheadrightarrow F^{n}_{1}/F^{n}_{2}\simeq
E^{1, n-1}_{\infty}\hookrightarrow E^{1, n-1}_{2}.
\end{equation}
Moreover, the morphism (\ref{spe}) is an isomorphism if $E^{l, m}_{2}=0$ for $l\geq 2$.

Let $i: \spec k\to \spec A$ be the closed immersion.
Since $k$ has $p$-cohomological dimension at most $1$,
\begin{equation*}
\operatorname{H}^{l}_{\et}
\left(
\spec A, 
\mathbf{R}^{m}j_{*}\mathbb{Z}/p^{r}(n)
\right)
=
\operatorname{H}^{l}_{\et}
\left(
\spec k, 
i^{*}\mathbf{R}^{m}j_{*}\mathbb{Z}/p^{r}(n)
\right)
=
0
\end{equation*}
for $l\geq 2$ by 
\cite[p.777, The proof of Proposition 2.2 b)]{Ge} 
and
\begin{align*}
&\operatorname{Ker}
\Bigl(
\operatorname{H}^{n+1}_{\et}
(\spec K, \mathbb{Z}/p^{r}(n))
\to 
\Gamma
\left(
\spec A,
\mathbf{R}^{n+1}j_{*}\mathbb{Z}/m(n)
\right)
\Bigr)\\
=&\operatorname{H}^{1}_{\et}
\left(
\spec A, \mathbf{R}^{n}j_{*}\mathbb{Z}/p^{r}(n)
\right)\\
=&
\operatorname{H}^{1}_{\et}
\left(
\spec k, i^{*}\mathbf{R}^{n}j_{*}\mathbb{Z}/p^{r}(n)
\right)
\end{align*}
by the spectral sequence
\begin{equation*}
\operatorname{H}^{l}_{\et}
\left(
\spec A, \mathbf{R}^{m}j_{*}\mathbb{Z}/p^{r}(n)
\right)
\Rightarrow
\operatorname{H}^{l+m}_{\et}
\left(
K, 
\mathbb{Z}/p^{r}(n)
\right)
\end{equation*}
and
\cite[p.777, The proof of Proposition 2.2 b)]{Ge}.
Moreover
\begin{align*}
\left(
i^{*}\mathbf{R}^{n}j_{*}\mathbb{Z}/p^{r}(n)
\right)_{\bar{\mathfrak{m}}}
=
\operatorname{H}^{n}_{\et}
\left(
\spec K_{\bar{\mathfrak m}}, \mathbb{Z}/p^{r}(n)
\right)
=
\operatorname{K}^{M}_{n}(K_{\bar{\mathfrak{m}}})/p^{r}
\end{align*}
by \cite[p.88, III, Theorem 1.15]{M} and \cite[p.131, Theorem (5.12)]{B-K}. 
Here $\bar{\mathfrak{m}}$ is a geometric point of
$\spec A$ such that 
$\kappa(\bar{\mathfrak{m}})$
is the separable closure of $k=\kappa(\mathfrak{m})$
and 
$\left(
i^{*}\mathbf{R}^{n}j_{*}\mathbb{Z}/p^{r}(n)
\right)_{\bar{\mathfrak{m}}}$
is the stalk of
$i^{*}\mathbf{R}^{n}j_{*}\mathbb{Z}/p^{r}(n)$
at $\bar{\mathfrak{m}}$.
Hence
\begin{align*}
&\operatorname{Ker}
\Bigl(
\operatorname{H}^{n+1}_{\et}
(\spec K, \mathbb{Z}/p^{r}(n))
\to 
\Gamma
\left(
\spec A,
\mathbf{R}^{n+1}j_{*}\mathbb{Z}/m(n)
\right)
\Bigr)\\
=&
\operatorname{H}^{1}
\left(
\operatorname{Gal}(k_{s}/k), 
\operatorname{K}^{M}_{n}(K_{\bar{\mathfrak{m}}})/p^{r}
\right).
\end{align*}
Therefore the statement follows from Lemma \ref{tru}.
\end{proof}

Let $A$ be a discrete valuation ring with the maximal ideal
$\mathfrak{m}$ and the fraction field $K$. Suppose that 
$\operatorname{char}(k)=p>0$.

Let $j: \spec K\to \spec A$ be the generic point,
$i: \spec k\to \spec A$ the closed immersion,
$k_{s}$ the separable closure
of $k$ and $K_{\bar{\mathfrak{m}}}$ the maximal unramified extension of $K$.

Then
\begin{align*}
&\operatorname{H}^{n}_{\et}
\left(
\spec A,
i_{*}\mathbb{Z}/p^{r}(n-1)
\right)\\
=&\operatorname{H}^{n}_{\et}
\left(
\spec k, 
\mathbb{Z}/p^{r}(n-1)
\right)\\
=&\operatorname{H}^{1}
\left(
\operatorname{Gal}(k_{s}/k), 
\operatorname{K}^{M}_{n-1}(k_{s})/p^{r}
\right)
\end{align*}
by \cite[Theorem 8.5]{Ge-L} 
and \cite[p.117, Corollary (2.8)]{B-K}.

\begin{prop}\upshape
Let the notations be same as above.
Suppose that $A$ is a henselian discrete valuation ring. 
Then 
the homomorphism
\begin{equation*}
\operatorname{H}^{n+1}_{\et}
\left(
\spec A, 
\tau_{\leq n}\left(
\mathbf{R}j_{*}\mathbb{Z}/p^{r}(n)
\right)
\right)
\to
\operatorname{H}^{n}_{\et}
\left(
\spec A, 
i_{*}\left(
\mathbb{Z}/p^{r}(n-1)
\right)
\right)
\end{equation*}
which is induced by the map
\begin{equation*}
\tau_{\leq n}
\left(
\mathbf{R}j_{*}\mathbb{Z}/p^{r}(n)_{\mathrm{\acute{e}t}}
\right)
\to
i_{*}\left(
\mathbb{Z}/p^{r}(n-1)_{\mathrm{\acute{e}t}}
\right)[-1]
\end{equation*}
agrees with the homomorphism
\begin{equation}\label{Ka}
\operatorname{H}^{1}
\left(
\operatorname{Gal}(k_{s}/k),
\mathrm{K}_{n}^{\mathrm{M}}
\left(K_{\bar{\mathfrak{m}}}
\right)
/p^{r}
\right)
\to
\operatorname{H}^{1}
\left(
\operatorname{Gal}(k_{s}/k),
\mathrm{K}_{n-1}^{\mathrm{M}}(k_{s})/p^{r}
\right)
\end{equation}
which is induced by the symbol map of the
Milnor $K$-group.

\end{prop}

\begin{proof}\upshape
We have the commutative diagram
\footnotesize
\begin{equation}
\begin{CD}
@. 
\operatorname{H}^{n+1}_{\et}
\left(
K, \mathbb{Z}/p^{r}(n)
\right)
\\
@. @|
\\
\operatorname{H}^{1}_{\et}
\left(
\spec A, 
\mathbf{R}^{n}j_{*}\mathbb{Z}/p^{r}(n)
\right)
@>>>
\operatorname{H}^{n+1}_{\et}
\left(
\spec A, 
\mathbf{R}j_{*}\mathbb{Z}/p^{r}(n)
\right)
\\
@VVV @VVV \\
\operatorname{H}^{1}_{\et}
\left(
\spec A, 
i_{*}\mathbf{R}^{n+1}i^{!}\mathbb{Z}/p^{r}(n)
\right)
@>>>
\operatorname{H}^{n+1}
_{\et}
\left(
\spec A, 
i_{*}\mathbf{R}i^{!}\mathbb{Z}/p^{r}(n)[+1]
\right)\\
@. @| \\
@.
\operatorname{H}^{n+2}_{\mathfrak{m}}
\left(
\spec A, 
\mathbb{Z}/p^{r}(n)
\right)
\end{CD}
\end{equation}
\normalsize
where the horizontal maps are given by spectral sequences

\begin{equation*}
\operatorname{H}^{l}_{\et}
\left(
\spec A, 
\mathbf{R}^{m}j_{*}\mathbb{Z}/p^{r}(n)
\right)
\Rightarrow
\operatorname{H}^{l+m}_{\et}
\left(
K, 
\mathbb{Z}/p^{r}(n)
\right)
\end{equation*}
and

\begin{equation*}
\operatorname{H}^{l}_{\mathrm{\acute{e}t}}
\left(
\spec A, 
i_{*}\mathbf{R}^{m}i^{!}\mathbb{Z}/p^{r}(n-1)
\right)
\Rightarrow
\operatorname{H}^{l+m}_{\mathfrak{m}}
\left(
\spec A, \mathbb{Z}/p^{r}(n)
\right).
\end{equation*}
%


Moreover, we have the commutative diagram
\begin{equation*}
\begin{CD}
\epsilon^{*}i^{*}\mathbf{R}j_{*}\mathbb{Z}/p^{r}(n)_{\mathrm{Zar}}
@>>>
\epsilon^{*}\mathbb{Z}/p^{r}(n-1)_{\mathrm{Zar}}[-1] \\
@VVV @VVV  \\
i^{*}\tau_{\leq n}
\left(
\mathbf{R}j_{*}\mathbb{Z}/p^{r}(n)_{\mathrm{\acute{e}t}}
\right)
@>>>
\tau_{\leq n+1}
\left(
\mathbf{R}i^{!}\mathbb{Z}/p^{r}(n)_{\mathrm{\acute{e}t}}
\right)
[+1]
\end{CD}
\end{equation*}
where the vertical maps are quasi-isomorphisms and the 
homomorphism
\begin{equation*}
\left(
\epsilon^{*}i^{*}\mathcal{H}^{n}
\left(
\mathbf{R}j_{*}\mathbb{Z}/p^{r}(n)_{\mathrm{Zar}}
\right)
\right)_{\bar{\mathfrak{m}}}
\to
\mathcal{H}^{n-1}\left(
\epsilon^{*}\mathbb{Z}/p^{r}(n-1)_{\mathrm{Zar}}
\right)_{\bar{\mathfrak{m}}}
\end{equation*}
agrees with the symbol map
\begin{math}
\mathrm{K}_{n}^{\mathrm{M}}
(
K_{\bar{\mathfrak{m}}}
)/p^{r}
\to
\mathrm{K}_{n-1}^{\mathrm{M}}
(
k_{s})
/p^{r}
\end{math}
by \cite[Lemma 3.2]{Ge-L}.
Therefore the statement follows.
\end{proof}

\begin{rem}\upshape
Let $A$ be a discrete valuation ring with the fraction field $K$ and
the residue field $k$. Suppose that 
$\operatorname{char}(k)=p>0$.

In the case where $A$ is henselian and
$[k: k^{p}]\leq n-1$, 
the homomorphism (\ref{Ka}) is defined by K.Kato (\cite[p.150, \S1, (1.3) (ii)]{K}).

In general, the homomorphism
\begin{equation*}
\HO^{n+1}_{\et}
\left(
K,
\mathbb{Z}/p^{r}(n)
\right)^{\prime}
\to
\HO^{n}_{\et}
\left(
k,
\mathbb{Z}/p^{r}(n-1)
\right)
\end{equation*}
in the exact sequence (\ref{dim1ex}) is 
the composition of homomorphisms
\begin{equation*}
\HO^{n+1}_{\et}
\left(
K,
\mathbb{Z}/p^{r}(n)
\right)^{\prime}
\to
\HO^{n+1}_{\et}
\left(
\tilde{K}_{\mathfrak{m}},
\mathbb{Z}/p^{r}(n)
\right)^{\prime}
\end{equation*} 
and (\ref{Ka})
where 
$\tilde{K}_{\mathfrak{m}}$ 
is the henselization of $K$.
\end{rem}
\section{Definition of 
a generalized Brauer group
}\label{defgb}

In this section, we define a generalization of the Brauer group.

\vspace{2.0mm}

Let 
\begin{math}
\epsilon: X_{\et} \to X_{\Zar}
\end{math}
be the canonical map of sites. Then we define  a generalization
of the Brauer group as follows.
\begin{defi}\upshape\label{DefGB}
Let $X$ be an equi-dimensional scheme. Then we define $\operatorname{H}^{n}_{\operatorname{B}}(X)$ as
\begin{equation*}
\operatorname{H}^{n}_{\operatorname{B}}(X)
=
\Gamma\left(
X,
\mathbf{R}^{n+1}\epsilon_{*}\mathbb{Z}
\left(
n-1
\right)_{\et}
\right).
\end{equation*}
%
%
%
%
%
%
\end{defi}
\begin{rem}\upshape
Let $X$ be an essentially
smooth scheme over a Dedekind domain
and $m$ a positive integer.
Since
\begin{equation*}
\mathbf{R}^{n}\epsilon_{*}(n-1)_{\et} 
\simeq
\mathcal{H}^{n}(\mathbb{Z}(n-1)_{\Zar})
=
0
\end{equation*}
by \cite[Theorem 1.2.2 and Corollary 4.4]{Ge} and \cite{V2}, 
we have
\begin{align*}
\HB^{n}(X)_{m}
\overset{\mathrm{def}}{=}&
\operatorname{Ker}
\left(
\HB^{n}(X)
\xrightarrow{\times m}
\HB^{n}(X)
\right)\\
=&
\Gamma
\left(
X,
\mathbf{R}^{n}\epsilon_{*}\mathbb{Z}/m(n-1)_{\et}
\right)
\end{align*}
by a distinguished triangle
\begin{align*}
\cdots
\to
\mathbf{R}\epsilon_{*}\mathbb{Z}(n-1)_{\et}
\xrightarrow{\times m}
\mathbf{R}\epsilon_{*}\mathbb{Z}(n-1)_{\et}
\to
\mathbf{R}\epsilon_{*}\mathbb{Z}/m(n-1)_{\et}
\to
\cdots.
\end{align*}
Therefore this cohomology group 
relates to Kato homology (cf. \cite[p.160, (5.3)]{KS}).
\end{rem}

For the following reason we can regard $\operatorname{H}^{n}_{\operatorname{B}}$ as a generalization of the 
Brauer group.
\begin{prop}\upshape\label{reason}
Let $X$ be an essentially smooth scheme 
over 
the spectrum of
a Dedekind
domain. Then
\begin{align*}
\operatorname{H}^{1}_{\operatorname{B}}(X)
=
\operatorname{H}^{1}_{\et}
\left(
X, \mathbb{Q}/\mathbb{Z}
\right)
~~\textrm{and}~~
\operatorname{H}^{2}_{\operatorname{B}}(X)
=
\operatorname{Br}(X)
\end{align*}
where $\operatorname{Br}(X)$ is the cohomological
Brauer group
$\operatorname{H}^{2}_{\et}
\left(X, \mathbb{G}_{m}\right)$.
\end{prop}
\begin{proof}\upshape
We prove
\begin{equation}\label{H2Br}
\operatorname{H}^{2}_{\operatorname{B}}(X)
=
\operatorname{Br}(X).
\end{equation}

Since $X$ is a smooth scheme of finite type over 
the spectrum of a 
Dedekind domain, there is a quasi-isomorphism
\begin{equation*}
\mathbb{Z}(1)_{\et}
\simeq 
\mathbb{G}_{m}[-1]
\end{equation*}
by \cite[pp.196--197]{Ge2}
and we have the morphism
\begin{equation}\label{brtr}
\operatorname{H}^{3}_{\et}
\left(
X, \mathbb{Z}(1)
\right)
\to
\Gamma
\left(
X, \mathbf{R}^{3}\epsilon_{*}\mathbb{Z}(1)_{\et}
\right)
\end{equation}
which is induced by the morphism
\begin{equation*}
\tau_{\leq 3}
\left(
\mathbf{R}\epsilon_{*}\mathbb{Z}(1)_{\et}
\right)
\to
\mathbf{R}^{3}\epsilon_{*}\mathbb{Z}(1)_{\et}.
\end{equation*}
Let $x\in X_{(0)}$ and $i_{x}: x\to X$ the closed immersion. Then the morphism
\begin{equation*}
\Gamma
\left(
X,
\mathbf{R}^{3}\epsilon_{*}\mathbb{Z}(1)_{\et}
\right)
\to
\prod_{x\in X_{(0)}}\Gamma
\left(
x, (i_{x})^{*}
\mathbf{R}^{3}\epsilon_{*}\mathbb{Z}(1)_{\et}
\right)
\end{equation*}
is injective and
\begin{equation*}
\Gamma
\left(
x, (i_{x})^{*}
\mathbf{R}^{3}\epsilon_{*}\mathbb{Z}(1)_{\et}
\right)
=
\operatorname{H}^{3}_{\et}
\left(
\spec\mathcal{O}_{X, x}, \mathbb{Z}(1)
\right)
\end{equation*}
by \cite[p.88, III, Proposition 1.13]{M}.
Hence 
\begin{equation*}
\Gamma
\left(
X,
\mathbf{R}^{3}\epsilon_{*}\mathbb{Z}(1)_{\et}
\right)
\subset
\bigcap_{x\in X_{(0)}}
\operatorname{H}^{3}_{\et}
\left(
\spec\mathcal{O}_{X, x}, \mathbb{Z}(1)
\right).
\end{equation*}
On the other hand,
\begin{equation*}
\operatorname{H}^{3}_{\et}
\left(
X, \mathbb{Z}(1)
\right)
=
\bigcap_{x\in X_{(0)}}
\operatorname{H}^{3}_{\et}
\left(
\spec\mathcal{O}_{X, x}, \mathbb{Z}(1)
\right)
\end{equation*}
by \cite[Remark 7.18]{SakD}.
Therefore we have an injective homomorphism
\begin{equation}\label{inv}
\Gamma\left(
X, 
\mathbf{R}^{3}\epsilon_{*}\mathbb{Z}(1)_{\et}
\right)    
\to
\operatorname{H}^{3}_{\et}
\left(
X, \mathbb{Z}(1)
\right)
\end{equation}
%
%
such that
\begin{equation*}
\operatorname{H}^{3}_{\et}
\left(
X, \mathbb{Z}(1)
\right)
\xrightarrow{(\ref{brtr})}
\Gamma\left(
X, 
\mathbf{R}^{3}\epsilon_{*}\mathbb{Z}(1)_{\et}
\right) 
\xrightarrow{(\ref{inv})} 
\operatorname{H}^{3}_{\et}
\left(
X, \mathbb{Z}(1)
\right)
=
\operatorname{id}
\end{equation*}
and the morphism (\ref{brtr}) is an isomorphism.
Hence the equation (\ref{H2Br}) follows.

Moreover we can show that
\begin{equation*}
\operatorname{H}^{1}_{\operatorname{B}}
\left(
X
\right)
=
\operatorname{H}^{1}_{\et}
\left(
X, \mathbb{Q}/\mathbb{Z}
\right)
\end{equation*}
as above. Therefore the statement follows.
\end{proof} 
In the following case, $\operatorname{H}^{n}_{\mathrm{B}}$ is expressed by \'{e}tale motivic cohomology.
\begin{prop}\upshape
\label{CaseGBEt}
Let $X$ be an essentially smooth scheme over 
the spectrum of
a Dedekind domain and
\begin{math}
\operatorname{H}^{i}_{\Zar}
\left(
X, \mathbb{Z}(n-1)
\right)
=
0
\end{math}
for $i\geq n+1$. 

Then
\begin{equation}\label{GeBrEteq}
\operatorname{H}^{n}_{\operatorname{B}}(X)
=
\operatorname{H}^{n+1}_{\et}
\left(
X, \mathbb{Z}(n-1)
\right).
\end{equation}
Especially,
if $A$ is a local ring of smooth algebra over a Dedekind domain
and $X=\spec A$, 
then the equation (\ref{GeBrEteq}) holds and
\begin{align}
\HB^{n}(X)_{m}
\overset{\mathrm{def}}{=}&
\operatorname{Ker}
\left(
\HB^{n}(X)
\xrightarrow{\times m}
\HB^{n}(X)
\right)
   \nonumber \\
\label{GBTm}
=&
\HO^{n}_{\et}
\left(
X,
\mathbb{Z}/m(n-1)
\right)
\end{align}
for any positive integer $m$.
\end{prop}
\begin{proof}\upshape
Since the canonical map induces a quasi-isomorphism
\begin{equation}\label{BZaret}
\mathbb{Z}(n-1)_{\Zar}
\simeq
\tau_{\leq n}
\mathbf{R}\epsilon_{*}\mathbb{Z}(n-1)_{\et}
\end{equation}
(\cite[p.774, Theorem 1.2.2]{Ge}, \cite{V2}), we have a distinguished triangle
\begin{align}\label{ZarEtdis}
\cdots
\to
\mathbb{Z}(n-1)_{\Zar}
\to
\tau_{\leq n+1}\mathbf{R}\epsilon_{*}\mathbb{Z}(n-1)_{\et} \\
\to
\mathbf{R}^{n+1}\epsilon_{*}\mathbb{Z}(n-1)_{\et}[-(n+1)]
\to 
\cdots. \nonumber
\end{align}
Hence the sequence
\begin{align}\label{exZaret}
0
&\to
\operatorname{H}^{n+1}_{\Zar}
\left(
X, \mathbb{Z}(n-1)
\right)
\to
\operatorname{H}^{n+1}_{\et}
\left(
X, \mathbb{Z}(n-1)
\right)
\to
\Gamma
\left(
X, \mathbf{R}^{n+1}\epsilon_{*}\mathbb{Z}(n-1)_{\et}
\right) \nonumber \\
&\to
\operatorname{H}^{n+2}_{\Zar}
\left(
X, \mathbb{Z}(n-1)
\right)
\to
\operatorname{H}^{n+2}_{\et}
\left(
X, \mathbb{Z}(n-1)
\right)
\end{align}
is exact.
Therefore the equation (\ref{GeBrEteq}) holds.

Assume that $A$ is a local ring of a smooth algebra over a Dedekind domain
and $X=\spec A$. Then
\begin{equation*}
\HO_{\Zar}^{i}
\left(
X, \mathbb{Z}(n-1)
\right)
=0
\end{equation*}
for $i\geq n$ by \cite[p.786, Corollary 4.4]{Ge} and
\begin{equation*}
\HO^{n}_{\et}
\left(
X,
\mathbb{Z}(n-1)
\right)
=
\HO^{n}_{\Zar}
\left(
X,
\mathbb{Z}(n-1)
\right)
=0
\end{equation*}
by the equation (\ref{BZaret}). Therefore the equations
(\ref{GeBrEteq}) and (\ref{GBTm}) hold. 
This completes the proof.
\end{proof}
\begin{rem}\upshape
%
%
In general,
\begin{equation*}
\operatorname{H}^{n+1}_{\et}
\left(
X, \mathbb{Z}(n-1)
\right)_{\operatorname{tor}}
\neq
\Gamma
\left(
X, \mathbf{R}^{n+1}\epsilon_{*}\mathbb{Z}(n-1)_{\et}
\right)_{\operatorname{tor}}.
\end{equation*}

Let $K$ be a field and $l$ a positive integer. Suppose that
$\mu_{l}\subset K$. 

Then we have
\begin{equation*}
\operatorname{H}^{5}_{\Zar}(\mathbf{P}^{m}_{K}, \mathbb{Z}(3))
=
\operatorname{H}^{1}_{\Zar}(\spec K, \mathbb{Z}(1))
=
K^{*}
\end{equation*}
for an integer $m\geq 2$
by the relation
\begin{equation*}
\operatorname{H}^{i}_{\Zar}
\left(
\mathbf{P}^{m}_{K}, 
\mathbb{Z}(n)
\right)
=
\displaystyle\bigoplus^{m}_{j=0}
\operatorname{H}^{i-2j}_{\Zar}
\left(
\spec K, \mathbb{Z}(n-j)
\right).
\end{equation*}
Hence
\begin{equation*}
\operatorname{H}^{5}_{\Zar}
\left(
\mathbf{P}^{m}_{K}, \mathbb{Z}(3)
\right)_{\operatorname{tor}}
\neq
0.
\end{equation*}
Since the sequence
\begin{align*}
0
\to&
\operatorname{H}^{5}_{\Zar}
\left(
\mathbf{P}^{m}_{K}, \mathbb{Z}(3)
\right)_{\operatorname{tor}}
\to
\operatorname{H}^{5}_{\et}
\left(
\mathbf{P}^{m}_{K}, \mathbb{Z}(3)
\right)_{\operatorname{tor}}  \\
\to&
\Gamma
\left(
X, 
\mathbf{R}^{5}\epsilon_{*}\mathbb{Z}(3)_{\et}
\right)_{\operatorname{tor}}
\end{align*}
is exact,
\begin{equation*}
\operatorname{H}^{5}_{\et}
\left(
\mathbf{P}^{m}_{K}, \mathbb{Z}(3)
\right)_{\operatorname{tor}}  
\neq
\Gamma
\left(
X, 
\mathbf{R}^{5}\epsilon_{*}\mathbb{Z}(3)
\right)_{\operatorname{tor}}.
\end{equation*}
%



%
\end{rem}
\begin{prop}\upshape
Let $X$ be an essentially smooth scheme over 
the spectrum of a 
Dedekind domain. Let 
$\alpha: X_{\et}\to X_{\Nis}$
be the canonical map of sites.
Then
\begin{align*}
\HB^{n+1}
\left(
X
\right)
=
\Gamma
\left(
X,
\mathbf{R}^{n+2}\alpha_{*}
\mathbb{Z}(n)_{\et}
\right).
\end{align*}
\end{prop}
\begin{proof}\upshape
Let 
$\beta: X_{\Nis} \to X_{\Zar}$ be the canonical map of sites.
Since $\beta^{*}$ is exact and
\begin{equation*}
\beta^{*}\mathbb{Z}(n-1)_{\Zar}
=
\mathbb{Z}(n-1)_{\Nis},
\end{equation*}
we have a quasi-isomorphism
\begin{equation*}
\mathbb{Z}(n-1)_{\Nis}
\simeq
\tau_{\leq n}
\mathbf{R}\alpha_{*}\mathbb{Z}(n-1)_{\et}
\end{equation*}
by (\cite[p.774, Theorem 1.2.2]{Ge}, \cite{V2}) and the sequence
\begin{align*}
0
&\to
\operatorname{H}^{n+1}_{\Nis}
\left(
X, \mathbb{Z}(n-1)
\right)
\to
\operatorname{H}^{n+1}_{\et}
\left(
X, \mathbb{Z}(n-1)
\right)
\to
\Gamma
\left(
X, \mathbf{R}^{n+1}\alpha_{*}\mathbb{Z}(n-1)_{\et}
\right) \nonumber \\
&\to
\operatorname{H}^{n+2}_{\Nis}
\left(
X, \mathbb{Z}(n-1)
\right)
\to
\operatorname{H}^{n+2}_{\et}
\left(
X, \mathbb{Z}(n-1)
\right)
\end{align*}
is exact. Moreover the sequence (\ref{exZaret}) is exact and
\begin{equation*}
\HO^{i}_{\Zar}
\left(
X,
\mathbb{Z}(n-1)
\right)
=
\HO^{i}_{\Nis}
\left(
X,
\mathbb{Z}(n-1)
\right)
\end{equation*}
for any $i$ by \cite[p.781, Proposition 3.6]{Ge}.
Therefore the statement follows from the five lemma.
\end{proof}
\begin{prop}\upshape
Let $X$ be an essentially smooth scheme over 
the spectrum of a 
Dedekind domain. Then
\begin{align}
\HB^{n+1}
\left(
X
\right)
&=
\Gamma
\left(
X,
\mathbf{R}^{n+1}\epsilon_{*}
\mathbb{Q/Z}(n)_{\et} 
\right)  \label{eptor}
\\
&=
\Gamma
\left(
X,
\mathbf{R}^{n+1}\alpha_{*}
\mathbb{Q/Z}(n)_{\et}
\right). \label{altor}
\end{align}
\end{prop}
\begin{proof}\upshape
We prove the equation (\ref{eptor}).
The sequence
\begin{align*}
\mathbf{R}^{n+1}\epsilon_{*}
\mathbb{Q}(n)_{\et}
\to
\mathbf{R}^{n+1}\epsilon_{*}
\mathbb{Q/Z}(n)_{\et}
\to
\mathbf{R}^{n+2}\epsilon_{*}
\mathbb{Z}(n)_{\et}
\to
\mathbf{R}^{n+2}\epsilon_{*}
\mathbb{Q}(n)_{\et}
\end{align*}
is exact. Thus,
the canonical map
\begin{equation*}
\mathbb{Q}(n)_{\Zar}
\xrightarrow{\sim}
\mathbf{R}\epsilon_{*}
\mathbb{Q}(n)_{\et}
\end{equation*}
is a quasi-isomorphism by \cite[p.781, Proposition 3.6]{Ge},
hence
\begin{equation*}
\mathbf{R}^{n+1}\epsilon_{*}
\mathbb{Q}(n)_{\et}
=
\mathbf{R}^{n+2}\epsilon_{*}
\mathbb{Q}(n)_{\et}
=
0
\end{equation*}
by \cite[p.786, Corollary 4.4]{Ge}. Therefore we have the equation 
(\ref{eptor}).
We can also prove the equation (\ref{altor}) as above.
\end{proof}
\section{Purity}\label{Puri}
First we show the exactness of the following sequence
in equi-characteristic cases.
\begin{prop}\upshape\label{equiGer}
Let $k$ be a field and
$\mathfrak{X}$ an essentially smooth scheme over 
$\spec k$. Suppose that $\mathfrak{X}$ is an
integral quasi-compact scheme.

Then the sequence
\begin{align}
0
\to
\HB^{n+1}
\left(
\mathfrak{X}
\right)
\to&
\operatorname{Ker}
\left(
\HO^{n+2}_{\et}
\left(
k(\mathfrak{X}),
\mathbb{Z}(n)
\right)
\to
\displaystyle
\prod_{x\in \mathfrak{X}^{(1)}}
\HO^{n+2}_{\et}
\left(
k(\mathcal{O}_{\mathfrak{X}, \bar{x}}),
\mathbb{Z}(n)
\right)
\right) \nonumber \\
\to&
\displaystyle
\bigoplus_{x\in \mathfrak{X}^{(1)}}
\HO^{n+1}_{\et}
\left(
\kappa(x),
\mathbb{Z}(n-1)
\right)  \label{equiex}
\end{align}
is exact.
\end{prop}
\begin{proof}\upshape
Let
\begin{math}
g: 
\spec
k(\mathfrak{X})
\to
\mathfrak{X}
\end{math}
be the generic point.

Since 
$\mathbf{R}^{n+1}\epsilon_{*}\mathbb{Z}/m(n)$
is the Zariski sheaf on $\mathfrak{X}$ associated
to the presheaf
\begin{equation*}
U
\mapsto 
\HO^{n+1}_{\et}
\left(
U, 
\mathbb{Z}/m(n)
\right)
\end{equation*}
for any positive integer $m$,
we have homomorphisms
\begin{align*}
\mathbf{R}^{n+1}\epsilon_{*}\mathbb{Z}/m(n)
\to
g_{*}
\left(
\HO^{n+1}_{\et}
\left(
k(\mathfrak{X}),
\mathbb{Z}/m(n)
\right)
\right)
\end{align*}
and
\begin{equation*}
g_{*}
\left(
\HO^{n+1}_{\et}
\left(
k(\mathfrak{X}),
\mathbb{Z}/m(n)
\right)
\right)^{\prime}
\to
\bigoplus_{x\in \mathfrak{X}^{(1)}}
(i_{x})_{*}
\left(
\HO^{n+1}_{\et}
\left(
\kappa(x),
\mathbb{Z}/m(n)
\right)
\right).
\end{equation*}

Here
\footnotesize
\begin{align*}
&
g_{*}
\left(
\HO^{n+1}_{\et}
\left(
k(\mathfrak{X}),
\mathbb{Z}/m(n)
\right)
\right)^{\prime}  \\
=
&\operatorname{Ker}
\left(
g_{*}
\left(
\HO^{n+1}_{\et}
\left(
k(\mathfrak{X}),
\mathbb{Z}/m(n)
\right)
\right)
\to
\prod_{x\in \mathfrak{X}^{(1)}}
(i_{x})_{*}
\left(
\HO^{n+1}_{\et}
\left(
k(\mathcal{O}_{\mathfrak{X}, x}),
\mathbb{Z}/m(n)
\right)
\right)
\right). 
\end{align*}
\normalsize
Then it suffices to show that the sequence
\begin{align*}
0
\to
\mathbf{R}^{n+1}\epsilon_{*}\mathbb{Z}/m(n)
\to&
g_{*}
\left(
\HO^{n+1}_{\et}
\left(
k(\mathfrak{X}),
\mathbb{Z}/m(n)
\right)
\right)^{\prime}  \\
\to&
\prod_{x\in \mathfrak{X}^{(1)}}
(i_{x})_{*}
\left(
\HO^{n+1}_{\et}
\left(
k(\mathcal{O}_{\mathfrak{X}, x}),
\mathbb{Z}/m(n)
\right)
\right)
\end{align*}
is exact. Hence it suffices to show that the sequence (\ref{equiex})
is exact in the case where
$\mathfrak{X}$
is the spectrum of a local ring of a smooth algebra over $k$. 

If $m\in k^{*}$, then
there is a quasi-isomorphism
\begin{equation*}
\mathbb{Z}/m(n)_{\et}
\xrightarrow{\sim}
\mu^{\otimes n}_{m}[0]
\end{equation*}
in
$D^{b}(\mathfrak{X}_{\et}, \mathbb{Z}/m\mathbb{Z})$
by \cite[Theorem 1.5]{G-L2}.
Here $\mu_{l}$ is the sheaf of
$l$-th roots of unity.

On the other hand, if $k$ has characteristic $p$, then
there is a quasi-isomorphism
\begin{equation*}
\mathbb{Z}/p^{r}(n)_{\et}
\xrightarrow{\sim}
\operatorname{W}_{r}\Omega^{n}_{\mathfrak{X}, \log}[-n]
\end{equation*}
in
$D^{b}(\mathfrak{X}_{\et}, \mathbb{Z}/p^{r}\mathbb{Z})$
for any positive integer $r$
by \cite[p.787, \S 5, (12)]{Ge}.
Here $\operatorname{W}_{r}\Omega^{n}_{\mathfrak{X}, \log}$ is 
the logarithmic de Rham-Witt sheaf.

Therefore the
statement follows from \cite[5.2. Theorem C]{P} and
(
Proposition \ref{supp}, 
\cite[p.600, Theorem 4.1]{Sh}).
\end{proof}
In the following, we consider $\HB^{n}(\mathfrak{X})$ in the case
where $\mathfrak{X}$ has mixed-characteristic.
\begin{thm}\upshape\label{mixinj}
Let $A$ be a discrete valuation ring of mixed-characteristic,
$R$ a local ring of a smooth 
algebra over $A$
and $K$ the quotient field of $R$.

Then the homomorphism
\begin{equation*}
\HO^{n+2}_{\et}
\left(
R,
\mathbb{Z}(n)
\right)
\to
\HO^{n+2}_{\et}
\left(
K,
\mathbb{Z}(n)
\right)
\end{equation*}
is injective.
\end{thm}
\begin{proof}\upshape
Let 
$Y \to \spec R$
be the inclusion of the closed fiber with open
complement
$j: U \to \spec R$.
Then $Y$ is the spectrum of a local ring of smooth
algebra over a field and
\begin{equation*}
\HO^{n}_{\et}
\left(
Y, 
\mathbb{Z}(n-1)
\right)
=
\HO^{n}_{\Zar}
\left(
Y,
\mathbb{Z}(n-1)
\right)
\end{equation*}
by (\cite[p.774, Theorem 1.2.2]{Ge}, \cite{V2}).
Moreover
\begin{equation*}
\HO^{n}_{\Zar}
\left(
Y,
\mathbb{Z}(n-1)
\right)
=0
\end{equation*}
by
\cite[p.779, Theorem 3.2]{Ge}
and 
\cite[p.786, Corollary 4.4]{Ge}. 
Hence
\begin{equation*}
\HO^{n+2}_{Y}
\left(
R_{\et}, 
\mathbb{Z}(n)
\right)
=
\HO^{n}_{\et}
\left(
Y, 
\mathbb{Z}(n-1)
\right)
=
0
\end{equation*}
by 
Proposition \ref{tru2}.
Therefore the homomorphism
\begin{equation*}
\HO^{n+2}_{\et}
\left(
R, 
\mathbb{Z}(n)
\right)
\to
\HO^{n+2}_{\et}
\left(
U, 
\mathbb{Z}(n)
\right)
\end{equation*}
is injective.

On the other hand,
\begin{align*}
\HO^{i}_{\Zar}
\left(
U,
\mathbb{Z}(n)
\right)
=
0
\end{align*}
for $i\geq n+2$ by 
\cite[p.779, Theorem 3.2]{Ge} and 
\cite[p.786, Corollary 4.4]{Ge}. 

Hence
\begin{equation*}
\HO^{n+2}_{\et}
\left(
U, 
\mathbb{Z}(n)
\right)
=
\HB^{n+1}
\left(
U
\right)
\end{equation*}
by Proposition \ref{CaseGBEt}
and the homomorphism
\begin{equation*}
\HB^{n+1}
\left(
U
\right)
\to
\HO^{n+2}_{\et}
\left(
K,
\mathbb{Z}(n)
\right)
\end{equation*}
is injective by Proposition \ref{equiGer}. 
Therefore the statement follows.
\end{proof}
\begin{cor}\upshape\label{intmax}
Let $A$ be a Dedekind domain of mixed-characteristic
and $\mathfrak{X}$  an essentially smooth scheme over $\spec A$. 

Then
\begin{align*}
\HB^{n+1}
\left(
\mathfrak{X}
\right)
=
\displaystyle
\bigcap_{x\in \mathfrak{X}_{(0)}}
\HB^{n+1}
\left(
\spec
\mathcal{O}_{\mathfrak{X}, x}
\right).
\end{align*}
\end{cor}
\begin{proof}\upshape
Let 
\begin{math}
j: \mathfrak{X}_{K}
\to
\mathfrak{X}
\end{math}
be the inclusion of the generic fiber
and
$i_{a}: \mathfrak{X}_{a}\to \mathfrak{X}$
the closed embedding of the fiber of 
$\mathfrak{X}$ 
over 
$a\in (\spec A)^{(1)}$. 
Then we have a distinguished
triangle
\begin{align*}
\cdots
\to
\bigoplus_{a\in (\spec A)^{(1)}}
i_{a*}\mathbf{R}i^{!}_{a}
\left(
\mathbf{R}^{n+1}\epsilon_{*}\mathbb{Z}/m(n)
\right)
\to
\mathbf{R}^{n+1}\epsilon_{*}\mathbb{Z}/m(n) \\
\to
\mathbf{R}j_{*}j^{*}
\left(
\mathbf{R}^{n+1}\epsilon_{*}\mathbb{Z}/m(n)
\right)
\to
\cdots
\end{align*}
by \cite[p.778, Lemma 2.4]{Ge}. On the other hand,
\begin{equation*}
\left(
j_{*}j^{*}
\left(
\mathbf{R}^{n+1}
\mathbb{Z}/m(n)
\right)
\right)_{x}
=
\Gamma
\left(
\left(
\spec
\mathcal{O}_{\mathfrak{X}, x}
\right)_{K},
\mathbf{R}^{n+1}
\epsilon_{*}
\mathbb{Z}/m(n)
\right)
\end{equation*}
by 
\cite[p.88, III, Proposition 1.13]{M} 
and the homomorphism
\begin{equation*}
\mathbf{R}^{n+1}\epsilon_{*}
\mathbb{Z}/m(n)
\to
j_{*}j^{*}
\left(
\mathbf{R}^{n+1}\epsilon_{*}
\mathbb{Z}/m(n)
\right)
\end{equation*}
is injective by Theorem \ref{mixinj}. Moreover, the homomorphism
\footnotesize
\begin{align*}
\Gamma
\left(
\mathfrak{X},
i_{a*}\mathbf{R}^{1}i^{!}_{a}
\left(
\mathbf{R}^{n+1}\epsilon_{*}
\mathbb{Z}/m(n)
\right)
\right)
\to
\displaystyle
\prod_{x\in 
\mathfrak{X}_{(0)}
}
\HO^{1}_{
\left(
\mathcal{O}_{\mathfrak{X}, x}
\right)_{a}
}
\left(
\mathcal{O}_{\mathfrak{X}, x},
\mathbf{R}^{n+1}
\epsilon_{*}\mathbb{Z}/m(n)
\right)
\end{align*}
\normalsize
is injective
where
\begin{math}
\left(
\mathcal{O}_{\mathfrak{X}, x}
\right)_{a}
=
\mathcal{O}_{\mathfrak{X}, x}
\otimes_{A}
A/a
\end{math}
. Hence the sequence
\footnotesize
\begin{align*}
0
\to
\HB^{n+1}
\left(
\mathfrak{X}
\right)
\to
\HB^{n+1}
\left(
\mathfrak{X}_{K}
\right)
\to
\displaystyle
\prod
_{
\substack{
x\in \mathfrak{X}_{(0)}\\
a\in 
(
\spec A
)_{(0)}
}
}
\HO^{1}_{
\left(
\mathcal{O}_{\mathfrak{X}, x}
\right)_{a}
}
\left(
\mathcal{O}_{\mathfrak{X}, x},
\mathbf{R}^{n+1}
\epsilon_{*}\mathbb{Z}/m(n)
\right)
\end{align*}
\normalsize
is exact and
\begin{align*}
\HB^{n+1}
\left(
\mathfrak{X}_{K}
\right)
=
\bigcap_{x\in
\left(
\mathfrak{X}_{K}
\right)_{(0)}
}
\HB^{n+1}
\left(
\spec \mathcal{O}_{\mathfrak{X}_{K}, x}
\right)
\end{align*}
by Proposition \ref{equiGer}. Therefore the statement follows.
\end{proof}
\begin{lem}\upshape\label{co2HBet}
Let 
$A$ be a discrete valuation ring of mixed-characteristic,
$R$ a local ring of a smooth algebra over 
$A$ and $X=\spec R$. 
Let $i: Z\to X$ be a regular closed subscheme of
codimension $2$ with open complement $j: U\to X$.
Suppose that 
$\operatorname{char}(Z)=p>0$.

Then
\begin{equation*}
\HO_{\Zar}^{i}
\left(
U,
\mathbb{Z}(n)
\right)
=
0
\end{equation*}
for $i\geq n+2$ and
\begin{equation}\label{co2}
\HB^{n}
\left(
U
\right)
=
\HO^{n+1}_{\et}
\left(
U,
\mathbb{Z}(n-1)
\right)
\end{equation}
\end{lem}
\begin{proof}\upshape
Let $Z=\spec R^{\prime}$. Then $R^{\prime}$ is a local ring of a regular
ring of finite type over a field. 
By Quillen's method 
(cf.\cite[\S7, The proof of Theorem 5.11]{Q}),
\begin{equation*}
R^{\prime}
=
\displaystyle
\lim_{\to}
R_{i}^{\prime}
\end{equation*}
where $R_{i}^{\prime}$ is a local ring of a smooth algebra over 
$\mathbb{F}_{p}$ and the maps $R_{i}^{\prime}\to R_{j}^{\prime}$
are flat.
Hence
\begin{equation}\label{equvan}
\HO^{i}_{\Zar}
\left(
Z,
\mathbb{Z}(n)
\right)
=
0
\end{equation}
for $i\geq n+1$ by \cite[p.786, Corollary 4.4]{Ge}. 
Therefore
\begin{equation*}
\HO^{i}_{\Zar}
\left(
U,
\mathbb{Z}(n)
\right)
=
0
\end{equation*}
by \cite[p.779, Theorem 3.2]{Ge} and we have the equation 
(\ref{co2})
by Proposition \ref{CaseGBEt}. This completes the proof.
\end{proof}
\begin{prop}\upshape\label{HBco2P}
With the notations of Lemma \ref{co2HBet}, 
we have
\begin{equation*}
\HB^{n}
\left(
X
\right)
=
\HB^{n}
\left(
U
\right).
\end{equation*}
\end{prop}
\begin{proof}\upshape
We have
\begin{equation*}
\HB^{n}
\left(
U
\right)
=
\displaystyle
\bigcap_{x\in U_{(0)}}
\HB^{n}
\left(
\mathcal{O}_{U, x}
\right)
\end{equation*}
by Corollary \ref{intmax}. 

Let $l$ be a positive integer which is prime to 
$\operatorname{char}(Z)=p>0$.
Suppose that $\mu_{l}\subset A$.
Then
\begin{align*}
&\HB^{n}
\left(
\mathcal{O}_{U, x}
\right)_{l}
\left(
\overset{\mathrm{def}}{=}
\operatorname{Ker}
\left(
\HB^{n}
\left(
\mathcal{O}_{U, x}
\right)
\xrightarrow{\times l}
\HB^{n}
\left(
\mathcal{O}_{U, x}
\right)
\right) 
\right) \\
=
&\displaystyle
\bigcap_{y\in U^{(1)}}
\HB^{n}
\left(
\mathcal{O}_{U, y}
\right)_{l}
\left(
=
\displaystyle
\bigcap_{y\in U^{(1)}}
\operatorname{Ker}
\left(
\HB^{n}
\left(
\mathcal{O}_{U, y}
\right)
\xrightarrow{\times l}
\HB^{n}
\left(
\mathcal{O}_{U, y}
\right)
\right)
\right)
\end{align*}
for $x\in U_{(0)}$ by 
\cite[p.774, Theorem 1.2.(2, 4, 5)]{Ge} and 
\cite{V2}. 

On the other hand, we have
\begin{equation*}
\HB^{n}
\left(
X
\right)_{l}
=
\displaystyle
\bigcap_{x\in X^{(1)}}
\HB^{n}
\left(
\mathcal{O}_{X, x}
\right)_{l}
\end{equation*}
by \cite[p.774, Theorem 1.2.(2, 4, 5)]{Ge} and 
\cite{V2}. 

Since
\begin{math}
X^{(1)}=U^{(1)},
\end{math}
we have
\begin{equation}\label{2purityl}
\HB^{n}
\left(
X
\right)_{l}
=
\HB^{n}
\left(
U
\right)_{l}.
\end{equation}
Even if that is the case where $\mu_{l}\not\subset A$,
we can show the equation (\ref{2purityl}) 
by a standard norm argument.
%
%
%

\vspace{2.0mm}

Therefore it is sufficient to prove that
\begin{equation*}
\HB^{n}
\left(
X
\right)_{p}
=
\HB^{n}
\left(
U
\right)_{p}
\end{equation*}
in the case where $A$ has mixed-characteristic $(0, p)$.

\vspace{2.0mm}

Assume that $A$ is a discrete valuation of mixed-characteristic
$(0, p)$. Then we have quasi-isomorphisms
\begin{equation*}
\tau_{\leq n+1}
\left(
\mathbb{Z}(n-3)^{Z}_{\et}[-4]
\right)
\xrightarrow{\sim}
\tau_{\leq n+1}
\mathbf{R}i^{!}
\mathbb{Z}(n-1)^{X}_{\et}
\end{equation*}
by Proposition \ref{tru2} and
\begin{equation*}
\mathbb{Z}/p(n-3)^{Z}_{\et}[-4]
\xrightarrow{\sim}
\tau_{\leq n+1}
\mathbf{R}i^{!}
\mathbb{Z}/p(n-1)^{X}_{\et}
\end{equation*}
by \cite[p.540, Theorem 4.4.7]{SaD} and 
\cite[p.187, Remark 3.7]{SaR}. 

Since the sequence
\begin{align*}
\HO^{n-3}_{\et}
\left(
Z, 
\mathbb{Z}(n-3)
\right)
\xrightarrow{\times p}
\HO^{n-3}_{\et}
\left(
Z, 
\mathbb{Z}(n-3)
\right)
\to
\HO^{n-3}_{\et}
\left(
Z, 
\mathbb{Z}/p(n-3)
\right)
\to 
0
\end{align*}
is exact 
by (\cite[p.774, Theorem 1.2.2]{Ge}, \cite{V2}) and the equation (\ref{equvan}), 
the sequence
\footnotesize
\begin{align*}
\HO^{n+1}_{Z}
\left(
X_{\et}, 
\mathbb{Z}(n-1)
\right)
\xrightarrow{\times p}
\HO^{n+1}_{Z}
\left(
X_{\et}, 
\mathbb{Z}(n-1)
\right)
\to
\HO^{n+1}_{Z}
\left(
X_{\et}, 
\mathbb{Z}/p(n-1)
\right)
\to 
0
\end{align*}
\normalsize
is exact 
and
we have
\begin{equation*}
\HO^{n+2}_{Z}
\left(
X_{\et},
\mathbb{Z}(n-1)
\right)_{p}
=
0.
\end{equation*}
Therefore we have
\begin{equation*}
\HO^{n+1}_{\et}
\left(
X,
\mathbb{Z}(n-1)
\right)_{p}
=
\HO^{n+1}_{\et}
\left(
U,
\mathbb{Z}(n-1)
\right)_{p}
\end{equation*}
and the statement follows from Lemma \ref{co2HBet} 
and Proposition \ref{CaseGBEt}.
\end{proof}
\begin{thm}\upshape\label{GBmixex}
Let 
$A$ be a Dedekind domain of mixed-characteristic
and
$\mathfrak{X}$ an essentially smooth scheme
over $\spec A$.

Then
\begin{equation}\label{HBPuri}
\HB^{n}
\left(
\mathfrak{X}
\right)
=
\displaystyle
\bigcap_{x\in \mathfrak{X}^{(1)}}
\HB^{n}
\left(
\mathcal{O}_{\mathfrak{X}, x}
\right)
\end{equation}
and the sequence
\footnotesize
\begin{align}\label{mixGerF}
0
\to
\HB^{n}
\left(
\mathfrak{X}
\right)
\to
\operatorname{Ker}
\Bigl(
\HB^{n}
\left(
k
\left(
\mathfrak{X}
\right)
\right)
\to
\displaystyle
\prod_{x\in \mathfrak{X}^{(1)}}
\HB^{n}
\left(
k
\left(
\mathcal{O}_{\mathfrak{X}, \bar{x}}
\right)
\right)
\Bigr)
\to
\displaystyle
\bigoplus_{x\in \mathfrak{X}^{(1)}}
\HB^{n-1}
\left(
\kappa(x)
\right)
\end{align}
\normalsize
is exact.
\end{thm}
\begin{proof}\upshape
By 
Proposition \ref{supp},
it suffices to prove the equation (\ref{HBPuri}).
We shall show the equation (\ref{HBPuri}) by induction
on $\operatorname{dim}(\mathfrak{X})$. In the case where
$\operatorname{dim}(\mathfrak{X})=1$, we can prove
the equations (\ref{HBPuri}) and (\ref{mixGerF}) by 
Proposition \ref{supp}
and Corollary \ref{intmax}.

Assume that the equation (\ref{HBPuri}) holds in the case 
where $\operatorname{dim}(\mathfrak{X})\leq a-1$.
Suppose that $\operatorname{dim}(\mathfrak{X})= a$.

Let 
$x\in \mathfrak{X}_{(0)}$ and
$i_{x}: Z_{x}\to \spec \mathcal{O}_{\mathfrak{X}, x}$ 
a regular closed subscheme of codimension $2$ with open
complement 
$j_{x}: U_{x}\to \spec \mathcal{O}_{\mathfrak{X}, x}$.
Then we have
\begin{equation*}
\HB^{n}
\left(
\mathcal{O}_{\mathfrak{X}, x}
\right)
=
\HB^{n}
\left(
U_{x}
\right)
\end{equation*}
by Proposition \ref{HBco2P}. Since 
\begin{math}
(\spec
\mathcal{O}_{\mathfrak{X}, x}
)^{(1)}
=
U^{(1)}_{x}
\end{math}
and
\begin{math}
\operatorname{dim}
(
U_{x}
)
=
\operatorname{dim}
(
\mathfrak{X}
)
-1,
\end{math}
we have
\begin{equation*}
\HB^{n}
\left(
U_{x}
\right)
=
\displaystyle
\bigcap_{
\substack
{
y\in \mathfrak{X}^{(1)} \\
x\in \Bar{\{y\}}
}
}
\HB^{n}
\left(
\mathcal{O}_{\mathfrak{X}, y}
\right)
\end{equation*}
by induction hypothesis. 
Therefore the equation (\ref{HBPuri}) follows from
Corollary \ref{intmax}
and the sequence (\ref{mixGerF}) is exact by 
Proposition \ref{supp}. 
This completes the proof.
\end{proof}
\begin{cor}\upshape\label{HenGer}
Let $R$ be a henselian local ring of a smooth
algebra over 
a discrete valuation ring of 
mixed-characteristic $(0,p)$.
Then the sequence
\begin{align*}
0
\to
\HO^{n+1}_{\et}
\left(
R,
\mathbb{Z}/p^{r}(n)
\right)  
\to&
\HO^{n+1}_{\et}
\left(
k(R),
\mathbb{Z}/p^{r}(n)
\right)^{\prime}  \\
\to& 
\displaystyle
\bigoplus_{
\substack
{
\mathfrak{p}\in \spec R  \\
\operatorname{ht}(\mathfrak{p})
=
1
}
}
\HO^{n}_{\et}
\left(
\kappa(\mathfrak{p}),
\mathbb{Z}/p^{r}(n-1)
\right)  
\end{align*}
is exact
for integers $n\geq 0$ and $r>0$
where     
\begin{align*}
&\HO^{n+1}_{\et}
\left(
k(R),
\mathbb{Z}/p^{r}(n)
\right)^{\prime}     \\
=
&\operatorname{Ker}
\left(
\HO^{n+1}_{\et}
\left(
k(R),
\mathbb{Z}/p^{r}(n)
\right)
\to
\displaystyle
\prod_{
\substack{
\mathfrak{p}\in \spec R  \\
\operatorname{ht}(\mathfrak{p})=1
}
}
\HO^{n+1}_{\et}
\left(
k(
R_{\bar{\mathfrak{p}}}
),
\mathbb{Z}/p^{r}(n)
\right)
\right)
\end{align*}
and
\begin{math}
R_{\bar{\mathfrak{p}}}
\end{math}
is the strictly henselization of
$R_{\mathfrak{p}}$.
\end{cor}
\begin{proof}\upshape
Let $A$ be a regular local ring with 
$\operatorname{dim}(A)\leq 1$ and
\begin{equation*}
R^{\prime}
=
\displaystyle\lim_{\substack{\to\\i\in I}}
A_{i}
\end{equation*}
where $A_{i}$
are $A$-algebras 
essentially of finite type over $A$
and
the maps $A_{i}\to A_{j}$ are \'{e}tale.

Then we have
\begin{equation*}
\HB^{n}(R^{\prime})
=
\displaystyle\lim_{\substack{\to\\i\in I}}
\HB^{n}(A_{i})
\end{equation*}
by \cite[pp.88--89, III, Lemma 1.16]{M}.
Therefore the statement follows from 
Theorem \ref{GBmixex} and Proposition \ref{CaseGBEt}.
\end{proof}
\section{\'{E}tale motivic cohomology of henselian regular local rings}
\label{Rigid}

\subsection{The equi-characteristic case}

The objective of this subsection is to prove Theorem \ref{Ihenmix} 
in the case 
where $R$ is a henselian local ring of a smooth algebra over 
a field of characteristic $p>0$
(Theorem \ref{hendR}). 
\begin{lem}\upshape\label{dRex}
Let $A$ be a regular local ring over $\mathbb{F}_{p}$.
Let $t$ be a regular element of $A$, $\mathfrak{n}=(t)$ and
$B=A/\mathfrak{n}$. Then we have an exact sequence
\begin{equation*}
0
\to
(
\mathfrak{n}\Omega^{i}_{A}+d\Omega^{i-1}_{A}
)
/d\Omega^{i-1}_{A}
\to
\Omega^{i}_{A}/d\Omega^{i-1}_{A}
\xrightarrow{\bar{g}_{i}}
\Omega^{i}_{B}/d\Omega^{i-1}_{B}
\to
0
\end{equation*}
where the homomorphism $\bar{g}_{i}$ is induced by the natural homomorphism
\begin{math}
g_{i}: \Omega^{i}_{A}
\xrightarrow{}
\Omega^{i}_{B} 
. 
\end{math}
\end{lem}
\begin{proof}\upshape
$A$ can be written as a filtering inductive limit
\begin{math}
\displaystyle\lim_{\substack{\to\\ \lambda}}A_{\lambda} 
\end{math}
of finitely generated smooth algebras over $\mathbb{F}_{p}$
by Popescu's theorem (\cite{Po}). Let 
$\mathfrak{n}_{\lambda}$ be an ideal $(t)$ of $A_{\lambda}$ and 
$B_{\lambda}=A_{\lambda}/(t)$. 
Then we may assume that $B_{\lambda}$ is $0$-smooth over 
$\mathbb{F}_{p}$.

By
\cite[p.194, Theorem 25.2]{Ma}, 
we have a split exact sequence
\begin{equation*}
0
\to
\mathfrak{n}_{\lambda}/\mathfrak{n}_{\lambda}^{2}
\xrightarrow{\delta_{1}}
\Omega_{A_{\lambda}}\otimes_{A_{\lambda}} B_{\lambda}
\xrightarrow{\alpha_{1}}
\Omega_{B_{\lambda}}
\to 
0
\end{equation*}
where the maps are given by 
\begin{equation*}
\delta_{1}\left(
\bar{a}
\right)
=
da\otimes 1
\end{equation*}
and
\begin{equation*}
\alpha_{1}(db\otimes \bar{c})
=
\bar{c}d\bar{b}
\end{equation*}
for $a\in \mathfrak{n}_{\lambda}$ and $b, c\in A_{\lambda}$.
Let $\gamma_{1}$ be a section of $\alpha_{1}$. Then we have
\begin{equation*}
\gamma_{1}(\bar{c}d\bar{b})-db\otimes \bar{c}
\in \operatorname{Im}(\delta_{1})
\end{equation*}
for $b, c\in A_{\lambda}$.
Therefore we have an exact sequence
\begin{equation*}
0
\to
\mathfrak{n}_{\lambda}/\mathfrak{n}_{\lambda}^{2}
\otimes_{B_{\lambda}}\Omega^{i-1}_{B_{\lambda}}
\xrightarrow{\delta_{i}}
\Omega^{i}_{A_{\lambda}}\otimes_{A_{\lambda}} B_{\lambda}
\xrightarrow{\alpha_{i}}
\Omega^{i}_{B_{\lambda}}
\to 
0
\end{equation*}
by \cite[p.284, Theorem C.2]{Ma}.
Here
\begin{align*}
\delta_{i} 
\left(
\bar{a}\otimes
(
d\bar{b}_{1}
\wedge\cdots\wedge
d\bar{b}_{i-1}
)
\right)
=&
\left(
da
\wedge
db_{1}
\wedge\cdots\wedge
db_{i-1}
\right)
\otimes 
\bar{1}
\end{align*}
for $a\in \mathfrak{n}_{\lambda}$, 
$b_{1},\cdots, b_{i-1}\in A_{\lambda}$
and
\begin{equation*}
\alpha_{i}
\left(
(
dc_{1}
\wedge\cdots\wedge
dc_{i}
)
\otimes
\bar{f}
\right)
=
\bar{f}
d\bar{c}_{1}
\wedge\cdots\wedge
d\bar{c}_{i}
\end{equation*}
for
\begin{math}
c_{1}, \cdots, c_{i}, f\in A_{\lambda}.
\end{math}
Hence the sequence
\begin{equation}\label{deRhamex}
0
\to
\mathfrak{n}/\mathfrak{n}^{2}
\otimes_{B}\Omega^{i-1}_{B}
\xrightarrow{\delta_{i}}
\Omega^{i}_{A}\otimes_{A} B
\xrightarrow{\alpha_{i}}
\Omega^{i}_{B}
\to 
0
\end{equation}
is exact.

On the other hand, we have
\begin{align*}
\operatorname{Ker}\left(
\alpha_{i}
\right)
=
\operatorname{Im}\left(
\delta_{i}
\right)
\subset
\left(
d\Omega^{i-1}_{A}+\mathfrak{n}\Omega^{i}_{A}
\right)
/
\mathfrak{n}\Omega^{i}_{A}
\end{align*}
by the sequence (\ref{deRhamex}). So we have
\begin{equation*}
\operatorname{Ker}\left(
g_{i}
\right)
\subset
d\Omega^{i-1}_{A}+\mathfrak{n}\Omega^{i}_{A}.
\end{equation*}
Since 
\begin{equation*}
d\Omega^{i-1}_{B}
=
\operatorname{Im}
\left(
d\Omega^{i-1}_{A}
\to
\Omega^{i}_{A}
\xrightarrow{g_{i}}
\Omega^{i}_{B} 
\right),
\end{equation*}
we have
\begin{equation*}
\operatorname{Ker}\left(
\bar{g}_{i}
\right)
\subset
\left(
d\Omega^{i-1}_{A}+\mathfrak{n}\Omega^{i}_{A}
\right)
/d\Omega^{i-1}_{A}.
\end{equation*}
Therefore the statement follows.
\end{proof}
\begin{lem}\upshape\label{homV}
Let $A$ be a 
regular local ring 
and
$X=\spec A$. Suppose that $\operatorname{char}(A)=p>0$. 
Then
\begin{equation*}
\operatorname{H}^{j}_{\et}
\left(
X,
Z\Omega^{i}_{X}
\right)
=
\operatorname{H}^{j}_{\et}
\left(
X, 
d\Omega^{i}_{X}
\right)
=
0
\end{equation*}
for all $j, i>0$.
Here 
\begin{math}
Z\Omega^{i}_{X}
=
\operatorname{Ker}
\left(
d: \Omega^{i}_{X}\to\Omega^{i+1}_{X}
\right).
\end{math}
\end{lem}
\begin{proof}\upshape
$Z\Omega^{i}_{X}$ and $d\Omega^{i}_{X}$
are locally free $\mathcal{O}_{X}$-modules
after twisting with Frobenius.
Hence
we have
\begin{equation*}
\operatorname{H}^{j}_{\Zar}
\left(
X,
Z\Omega^{i}_{X}
\right)
=
\operatorname{H}^{j}_{\Zar}
\left(
X,
d\Omega^{i}_{X}
\right)
=
0
\end{equation*}
for all $j>0$ by \cite[p.103, III, Lemma 2.15]{M}. 
On the other hand,
\begin{align*}
 \Gamma(U, Z\Omega^{i}_{X}\otimes_{\mathcal{O}_{X}}U)
 =
 Z\Omega^{i}_{U}
 ~\textrm{and}~
 \Gamma(U, d\Omega^{i}_{X}\otimes_{\mathcal{O}_{X}}U)
 =
 d\Omega^{i}_{U}
\end{align*}
for $U/X$ \'{e}tale 
by \cite[p.48, II, Proposition 1.3]{M}
and \cite[p.574, Proposition 2.5]{Sh}.
Therefore
we have
\begin{equation*}
\operatorname{H}^{j}_{\et}
\left(
X,
Z\Omega^{i}_{X}
\right)
=
\operatorname{H}^{j}_{\et}
\left(
X,
d\Omega^{i}_{X}
\right)
=
0 
\end{equation*}
by  
\cite[p.114, III, Remark 3.8]{M}. 
This completes the proof.
\end{proof}
\begin{thm}\upshape
\label{hendR}
Let $A$ be a henselian regular local ring
over $\mathbb{F}_{p}$ and $k$ the residue field
of $A$. Then the homomorphism
\begin{equation}\label{rigpm}
\operatorname{H}^{1}_{\et}
\left(
\spec A, 
\Omega^{i}_{A, \log}
\right)
\to
\operatorname{H}^{1}_{\et}
\left(
\spec k, 
\Omega^{i}_{k, \log}
\right)
\end{equation}
is an isomorphism.

Suppose that $A$ is a henselian local of smooth algebra
over
a field of characteristic $p>0$.
Then we have an isomorphism
\begin{equation*}
\HO^{i+1}_{\et}
\left(
A,
\mathbb{Z}/p(i)
\right)
\xrightarrow{\sim}
\HO^{i+1}_{\et}
\left(
k,
\mathbb{Z}/p(i)
\right)
\end{equation*}
by the isomorphism (\ref{rigpm}) and \cite[p.787, \S 5, (12)]{Ge}.
 
\end{thm}
\begin{proof}\upshape
Let $t$ a regular element of $A$ and $B=A/(t)$. 
Then it suffices to show that the homomorphism
\begin{equation}\label{pos}
\operatorname{H}^{1}_{\et}
\left(
\spec A, 
\Omega^{i}_{A, \log}
\right)
\to
\operatorname{H}^{1}_{\et}
\left(
\spec B, 
\Omega^{i}_{B, \log}
\right)
\end{equation}
is an isomorphism.

We have the following commutative diagram.
\begin{equation}\label{dR}
\begin{CD}
\operatorname{Ker}(g_{i})
@.\to
(
\mathfrak{n}\Omega^{i}_{A}+d\Omega^{i-1}_{A}
)
/d\Omega^{i-1}_{A}
\\
@VVV @VVV\\
\Omega^{i}_{A}
@.\xrightarrow{~~~1-F~~~}
\Omega^{i}_{A}/d\Omega^{i-1}_{A}
@.\xrightarrow{~~~}
\operatorname{H}^{1}_{\et}
\left(
\spec A, 
\Omega^{i}_{A, \log}
\right)
@.\to
0\\
@V{g_{i}}VV @VVV @VVV \\
\Omega^{i}_{B}
@.\xrightarrow{~~~1-F~~~}
\Omega^{i}_{B}/d\Omega^{i-1}_{B}
@.\xrightarrow{~~~}
\operatorname{H}^{1}_{\et}
\left(
\spec B, 
\Omega^{i}_{B, \log}
\right)    \\
@VVV \\
0
\end{CD}
\end{equation}
where $F$ is the homomorphism which is induced by the Frobenius operator and $\mathfrak{n}=(t)$. 

Then the horizontal arrows in (\ref{dR}) are exact by 
\cite[p.576, Proposition 2.8]{Sh} and Lemma \ref{homV}.
Moreover the vertical arrow in (\ref{dR}) is exact by Lemma \ref{dRex}. 

If the upper homomorphism in (\ref{dR})
is surjective, we can show that
the homomorphism (\ref{pos}) is injective
by chasing diagram (\ref{dR}). We have surjective homomorphism
\begin{align*}
A\otimes(A^{*})^{\otimes i}
&\to
\Omega^{i}_{A}  \\
a\otimes b_{1}\otimes \cdots \otimes b_{i}
&\mapsto
a\frac{db_{1}}{b_{1}}\wedge\cdots\wedge\frac{db_{i}}{b_{i}}
\end{align*}
by \cite[p.122, Lemma (4.2)]{B-K} and
\begin{equation*}
F\left(
a\frac{db_{1}}{b_{1}}\wedge\cdots\wedge\frac{db_{i}}{b_{i}}
\right)
=
a^{p}\frac{db_{1}}{b_{1}}\wedge\cdots\wedge\frac{db_{i}}{b_{i}}.
\end{equation*}
Since
\begin{equation*}
\mathfrak{n}\Omega^{i}_{A}
\subset
\operatorname{Ker}(g_{i}),
\end{equation*}
it suffices to show that for any $a\in \mathfrak{n}$
there exists a $b\in \mathfrak{n}$ such that
\begin{equation}\label{sol}
b^{p}-b=a. 
\end{equation}
By the definition of Henselian,
there exists a $b\in A\setminus A^{*}$ such that $b$ is a solution of the 
equation (\ref{sol}) and 
\begin{math}
b+1,\cdots, b+p-1\in A^{*} 
\end{math}
are also solutions of the equation (\ref{sol}).
Hence $b\in \mathfrak{n}$ by the equation (\ref{sol}). 
Therefore  the homomorphism (\ref{pos}) is injective.
Moreover the homomorphism (\ref{pos}) is surjective
by the diagram (\ref{dR}). 
This completes the proof.
\end{proof}
\subsection{The mixed-characteristic case}

In this subsection we show Theorem \ref{Ihenmix} 
in the case 
where $R$ is mixed-characteristic (Theorem \ref{henmix}).

Let $A$ be a mixed-characteristic henselian discrete valuation ring, $K$ its fraction field 
and $\pi$ a prime element of $A$. 

We consider the following diagram of schemes.
\begin{equation*}
\begin{CD}
\mathfrak{X}\otimes_{A}K
@>{j}>>
\mathfrak{X}
@<{i}<<
Y=\mathfrak{X}\otimes_{A}A/(\pi) \\
@VVV @VVV @VVV \\
\spec K
@>>> 
\spec A
@<<<
\spec A/(\pi)
\end{CD}
\end{equation*}
where the vertical arrows are smooth.

\vspace{2.0mm}

Suppose that
$\operatorname{char}(A/(\pi))=p>0$. Then
the filtration 
\begin{math}
\operatorname{U}^{m}
\operatorname{M}^{n}_{r}
\end{math}
of
\begin{equation*}
\operatorname{M}^{n}_{r}
=
i^{*}\mathbf{R}^{n}j_{*}
\mu_{p^{r}}^{\otimes n}
\end{equation*}
is defined and the structure of
\begin{math}
\operatorname{M}^{n}_{r}
\end{math}
is studied in \cite{B-K}.

\vspace{2.0mm}

In particular, the structure of
$\operatorname{M}^{n}_{1}$
is as follows.

\begin{thm}\upshape(\cite[p.112, Corollary (1.4.1)]{B-K})
\label{BKmix}
Let $e$ be the absolute ramification index of $K$ and
\begin{math}
e^{\prime}
=
\frac{ep}{p-1}.
\end{math}

Then the sheaf $\operatorname{M}^{n}_{1}$ has the following structure.
\begin{itemize}
\item[(i)] 
\begin{equation*}
\operatorname{M}^{n}_{1}/
\operatorname{U}^{1}\operatorname{M}^{n}_{1}
\simeq
\Omega^{n}_{Y, \log}
\oplus
\Omega^{n-1}_{Y, \log}.
\end{equation*}
\item[(ii)]
If $1\leq m < e^{\prime}$ and $m$ is prime to $p$,
\begin{equation*}
\operatorname{U}^{m}
\operatorname{M}^{n}_{1}/
\operatorname{U}^{m+1}
\operatorname{M}^{n}_{1}
\simeq
\Omega^{n-1}_{Y}.
\end{equation*}
\item[(iii)] If $1\leq m <e^{\prime}$ and $p|m$,
\begin{align*}
\operatorname{U}^{m}
\operatorname{M}^{n}_{1}/
\operatorname{U}^{m+1}
\operatorname{M}^{n}_{1}
\simeq
&
\operatorname{B}^{n-1}_{1}
\oplus
\operatorname{B}^{n-2}_{1} 
\end{align*}
where
\begin{equation*}
\operatorname{B}^{q}_{1}
=
\operatorname{Image}
\left(
d:
\Omega^{q-1}_{Y}
\to
\Omega^{q}_{Y}
\right)
\end{equation*}
for an integer $q$.
\item[(iv)] For $m\geq e^{\prime}$,
\begin{equation*}
\operatorname{U}^{m}
\operatorname{M}^{n}_{1}
=
0.
\end{equation*}
\end{itemize}
\end{thm}

\cite[p.548, Corollary 1.7]{H} and 
\cite[pp.184--185, Theorem 3.3]{SaR}
are the improved versions of Theorem \ref{BKmix}. 

As an application, we have the following lemma.
\begin{lem}\upshape\label{Uvan}
Let $R$ be a henselian local ring of a smooth scheme over
a mixed-characteristic discrete valuation ring $A$
and $\pi$ a prime element of $A$. 

Then we have
\begin{equation*}
\HO^{q}_{\et}
\left(
R/(\pi),
\operatorname{U}^{1}
\operatorname{M}_{1}^{n}
\right)
=
0
\end{equation*}
for $q\geq 1$.
\end{lem}
\begin{proof}\upshape
We may assume that $A$ is a henselian discrete valuation ring.
Let $q\geq 1$.
Since $R/(\pi)$ is a henselian local ring and 
$\operatorname{char}(R/(\pi))>0$, we have
\begin{equation*}
\HO^{q}_{\et}
\left(
R/(\pi),
\operatorname{B}_{1}^{n-1}
\right)
=
\HO^{q}_{\et}
\left(
R/(\pi),
\operatorname{B}_{1}^{n-2}
\right)
=
0
\end{equation*}
by Lemma \ref{homV}. Moreover
\begin{equation*}
\HO^{q}_{\et}
\left(
R/(\pi),
\Omega^{n-1}_{R/(\pi)}
\right)
=
0
\end{equation*}
by \cite[p.103, III, Lemma 2.15]{M} and 
\cite[p.114, III, Remark 3.8]{M}.
Therefore the statement follows from Theorem \ref{BKmix}.
\end{proof}
We prove Theorem \ref{Ihenmix} 
by computing the cohomology group
\begin{equation*}
\HO^{n+1}_{\et}
\left(
R/(\pi),
i_{*}
\tau_{\leq n}
\mathbf{R}j_{*}\mu_{p}^{\otimes n}
\right)
\end{equation*}
as follows.
\begin{thm}\upshape\label{henmix}
Let $R$ be a henselian local ring of a smooth scheme over
a mixed-characteristic discrete valuation ring $A$ and
$k$ the residue field of $R$. 

Then we have an isomorphism
\begin{equation*}
\HO^{n+1}_{\et}
\left(
R,
\mathbb{Z}/m(n)
\right)
\xrightarrow{\sim}
\HO^{n+1}_{\et}
\left(
k,
\mathbb{Z}/m(n)
\right)
\end{equation*}
for any positive integer $m$.
\end{thm}
\begin{proof}\upshape

Let 
$\pi$ be a prime element of $A$ 
and
$i: \spec R/(\pi)\to \spec R$ 
the closed subscheme with open
complement
$j: \spec R[\pi^{-1}]\to \spec R$.

Then the restriction map
\begin{equation*}
\HO^{n+1}_{\et}
\left(
R,
\mathbb{Z}/m(n)
\right)
\xrightarrow{\sim}
\HO^{n+1}_{\et}
\left(
R/(\pi),
i^{*}\mathbb{Z}/m(n)
\right)
\end{equation*}
is an isomorphism
by \cite[p.777, The proof of Proposition 2.2.b)]{Ge}.
Therefore it suffices to show that the homomorphism
\begin{equation}\label{mixizz}
\HO^{n+1}_{\et}
\left(
R/(\pi),
i^{*}\mathbb{Z}/m(n)
\right)
\to
\HO^{n+1}_{\et}
\left(
R/(\pi),
\mathbb{Z}/m(n)
\right)
\end{equation}
is an isomorphism in the case where 
$m=\operatorname{char}(k)=p>0$
by Theorem \ref{hendR}.

\vspace{2.0mm}

We consider the cohomology group
\begin{math}
\HO^{n+1}_{\et}
\left(
R/(\pi),
i_{*}
\tau_{\leq n}
\mathbf{R}j_{*}\mu_{p}^{\otimes n}
\right).
\end{math}

\vspace{2.0mm}

We have the spectral sequence
\begin{equation}\label{spemix}
\HO^{s}_{\et}
\left(
R/(\pi),
\mathcal{H}^{t}
\left(
i^{*}
\tau_{\leq n}
\mathbf{R}j_{*}\mu^{\otimes n}_{p}
\right)
\right)
\Rightarrow
\mathbb{R}^{s+t}
\Gamma_{\et}
\left(
i^{*}
\tau_{\leq n}
\mathbf{R}j_{*}\mu^{\otimes n}_{p}
\right)
\end{equation}
and
\begin{equation*}
\mathbb{R}^{n+1}
\Gamma_{\et}
\left(
i^{*}
\tau_{\leq n}
\mathbf{R}j_{*}\mu^{\otimes n}_{p}
\right)
=
\HO^{n+1}_{\et}
\left(
R/(\pi),
i^{*}
\tau_{\leq n}
\mathbf{R}j_{*}\mu^{\otimes n}_{p}
\right)
\end{equation*}
where
\begin{math}
\mathbb{R}^{*}
\Gamma_{\et}
\end{math}
is the right hyper-derived functor of the global sections functor
$\Gamma_{\et}$ from 
$\mathbb{S}_{\left(\spec R/(\pi)\right)_{\et}}$. 

Since
$p$-cohomological dimension of $k$
is at most $1$,
\begin{equation}\label{mixvan}
\HO^{s}_{\et}
\left(
R/(\pi),
i^{*}\mathbf{R}^{t}j_{*}\mu_{p}^{\otimes n}
\right)
=0
\end{equation}
for $s\geq 2$ by 
\cite[p.777, The proof of Proposition 2.2.b)]{Ge}. Hence we have
\begin{align*}
\HO^{n+1}_{\et}
\left(
R/(\pi),
\tau_{\leq n}
\left(
i^{*}\mathbf{R}j_{*}\mu_{p}^{\otimes n}
\right)
\right)
=
\HO^{1}_{\et}
\left(
R/(\pi),
i^{*}\mathbf{R}^{n}j_{*}\mu_{p}^{\otimes n}
\right)
\end{align*}
by the spectral sequence (\ref{spemix}) and
\begin{equation*}
\HO^{1}_{\et}
\left(
R/(\pi),
i^{*}\mathbf{R}^{n}j_{*}\mu_{p}^{\otimes n}
\right)
=
\HO^{1}_{\et}
\left(
R/(\pi),
\operatorname{M}^{n}_{1}/
\operatorname{U}^{1}\operatorname{M}^{n}_{1}
\right)
\end{equation*}
by Lemma \ref{Uvan}. Therefore we have
\begin{align}
&\HO^{n+1}_{\et}
\left(
R/(\pi),
\tau_{\leq n}
\left(
i^{*}\mathbf{R}j_{*}\mu_{p}^{\otimes n}
\right)
\right)    \nonumber  \\
\label{mixdec}
=
&\HO^{n+1}_{\et}
\left(
R/(\pi),
\mathbb{Z}/p(n)
\right)
\oplus
\HO^{n}_{\et}
\left(
R/(\pi),
\mathbb{Z}/p(n-1)
\right)
\end{align}
by Theorem \ref{BKmix} (i). On the other hand, the homomorphism
\begin{equation*}
\HO^{n}_{\et}
\left(
R/(\pi),
i^{*}\tau_{\leq n}
\mathbf{R}j_{*}\mu^{\otimes n}_{p}
\right)
\to
\HO^{0}_{\et}
\left(
R/(\pi),
i^{*}\mathbf{R}^{n}j_{*}\mu^{\otimes n}_{p}
\right)
\end{equation*}
is surjective by 
the spectral sequence (\ref{spemix}) and 
the equation (\ref{mixvan}).
Moreover the homomorphism
\begin{equation*}
\HO^{0}_{\et}
\left(
R/(\pi),
i^{*}\mathbf{R}^{n}j_{*}\mu^{\otimes n}_{p}
\right)
\to
\HO^{0}_{\et}
\left(
R/(\pi),
\operatorname{M}^{n}_{1}/
\operatorname{U}^{1}
\operatorname{M}^{n}_{1}
\right)
\end{equation*}
is surjective by Lemma \ref{Uvan}. 
Therefore the homomorphism
\begin{equation*}
\HO^{n}_{\et}
\left(
R/(\pi),
i^{*}\tau_{\leq n}
\mathbf{R}j_{*}\mu^{\otimes n}_{p}
\right)
\to
\HO^{n-1}_{\et}
\left(
R/(\pi),
\mathbb{Z}/p(n-1)
\right)
\end{equation*}
is surjective and the sequence
\begin{align}
0
\to
\HO^{n+1}_{\et}
\left(
R/(\pi),
i^{*}\mathbb{Z}/p(n)
\right)
\to
\HO^{n+1}_{\et}
\left(
R/(\pi),
i^{*}\tau_{\leq n}
\mathbf{R}j_{*}\mu^{\otimes n}_{p}
\right)   \nonumber \\
\label{seqiz}
\to
\HO^{n}_{\et}
\left(
R/(\pi),
\mathbb{Z}/p(n-1)
\right)
\to
0
\end{align}
is exact by the distinguished triangle
\begin{align*}
\cdots
\to
i^{*}\mathbb{Z}/p(n)_{\et}
\to
i^{*}\tau_{\leq n}\mathbf{R}j_{*}\mathbb{Z}/p(n)_{\et}
\to
\mathbb{Z}/p(n)_{\et}
\to
\cdots.
\end{align*}
Therefore the homomorphism (\ref{mixizz}) is an isomorphism 
by (\ref{mixdec}) and (\ref{seqiz}). This completes the proof.
\end{proof}
\section{Local-global principle}\label{App}
%
\begin{prop}\upshape\label{RHGL}
Let 
$A$ be a discrete valuation ring of mixed characteristic
$(0, p)$ and 
$\pi$ a prime element of $A$.
Let $R$ be a henselian local ring of a smooth algebra over
$A$.  Then the map
\footnotesize
\begin{equation}\label{LGL}
\HO^{n+1}_{\et}
\left(
k(R),
\mu^{\otimes n}_{p^{r}}
\right)
\to
\HO^{n+1}_{\et}
\left(
k(
\tilde{
R_{(\pi)}
}
),
\mu^{\otimes n}_{p^{r}}
\right)
\oplus
\displaystyle
\bigoplus_{
\substack{
\mathfrak{p}\in \left(\spec R\right)^{(1)} 
\backslash (\pi)
}
}
\HO^{n+2}_{\mathfrak{p}}
\left(
(
\tilde{
R_{\mathfrak{p}}
}
)_{\et},
\mu^{\otimes n}_{p^{r}}
\right)
\end{equation}
\normalsize
is injective
where $\tilde{R_{\mathfrak{p}}}$ is the henselization of $R$ at 
$\mathfrak{p}$.
\end{prop}
\begin{proof}\upshape
By Theorem \ref{GBmixex},
it suffices to show that the homomorphism
\begin{equation}\label{LArt}
\HO^{n+1}_{\et}\left(
R,
\mathbb{Z}/p^{r}(n)
\right)
\to
\HO^{n+1}_{\et}\left(
\tilde{R_{(\pi)}},
\mathbb{Z}/p^{r}(n)
\right)
\end{equation}
is injective.

We consider the commutative diagram
\begin{equation}\label{comGL}
\begin{CD}
\HO^{n+1}_{\et}\left(
R,
\mathbb{Z}/p^{r}(n)
\right)
@>{(\ref{LArt})}>>
\HO^{n+1}_{\et}\left(
\tilde{R_{(\pi)}},
\mathbb{Z}/p^{r}(n)
\right)\\
@VVV @VVV  \\
\HO^{n+1}_{\et}\left(
R/(\pi),
\mathbb{Z}/p^{r}(n)
\right)
@>>>
\HO^{n+1}_{\et}\left(
\tilde{R_{(\pi)}}/(\pi),
\mathbb{Z}/p^{r}(n)
\right).
\end{CD}
\end{equation}
Then the left map in the diagram (\ref{comGL})
is an isomorphism by Theorem \ref{henmix}. 
Moreover
the lower map in the diagram (\ref{comGL}) is injective
by Theorem \ref{mixinj} and
$R_{(\pi)}/(\pi)
=
\tilde{R_{(\pi)}}/(\pi)$. Therefore the homomorphism
(\ref{LArt})
is injective. This completes the proof.
\end{proof}
We review the $K$-theoretic fact
before we prove the main result of this section (Theorem \ref{LGG}).

For a field $F$ and an integer $n\geq 1$,
$\operatorname{K}^{M}_{n}(F)$ denotes 
the $n$-th Milnor $K$-group
of $F$. 
Then we have the following fact:
\begin{lem}\upshape(cf.\cite[Lemma 2.1]{Hu})
\label{WA}
Let $v_{1}, \cdots, v_{s}$ be a finite collection 
of independent discrete valuations on a field $F$ of
characteristic $0$. Denote by $F_{i}$
the henselization of $F$ at $v_{i}$ for each $i$.
Let $r\geq 1$ be an integer. Then for every $n\geq 1$,
the natural map
\begin{equation*}
\operatorname{K}^{M}_{n}(F)/r
\to
\bigoplus_{i}
\operatorname{K}^{M}_{n}(F_{i})/r
\end{equation*}
is surjective.
\end{lem}

\begin{proof}\upshape
Since the sequence
\begin{equation*}
\operatorname{K}^{M}_{n}(F)/r
\xrightarrow{\times r^{\prime}}
\operatorname{K}^{M}_{n}(F)/r r^{\prime}
\to
\operatorname{K}^{M}_{n}(F)/r^{\prime}
\to 
0
\end{equation*}
is exact, it suffices to show the statement 
in the case where $r$ is a prime number $p$.

Let
\begin{equation*}
\operatorname{U}_{F_{i}}^{m}
=
\{
x\in F_{i}, v_{i}(1-x)
\geq
m
\}
\end{equation*}
for $m\geq 1$.
If
\begin{equation*}
\operatorname{U}_{F_{i}}^{N}
\subset
(F_{i}^{*})^{p}
\end{equation*}
for any $i$ and sufficient large $N$,
we can show the statement
by the weak approximation property. 

Let $\kappa(v_{i})$ be the residue field of $v_{i}$. 
In the case where
$(\operatorname{char}(\kappa(v_{i})), p)=1$,
we have
\begin{equation*}
\operatorname{U}_{F_{i}}^{1}
\subset
(F_{i}^{*})^{p}.
\end{equation*}
In the case where
$\operatorname{char}(\kappa(v_{i}))=p$,
we have
\begin{equation*}
\operatorname{U}_{F_{i}}^{m}
\subset
(F_{i}^{*})^{p}
\end{equation*}
for $m>\frac{v_{i}(p)\cdot p}{p-1}$
(cf. \cite[p.124, Lemma (5.1)]{B-K}). This completes the proof.

\end{proof}
\begin{thm}\upshape\label{LGG}
Let $A$ be an excellent henselian discrete valuation ring
of mixed characteristic
$(0, p)$ and 
$\pi$ a prime element of $A$.
Let $\mathfrak{X}$ be 
a connected proper smooth curve over $\spec A$,
$K$ the fraction field of $\mathfrak{X}$
and
$K_{(\eta)}$ the fraction field of
the henselization of
$\mathcal{O}_{\mathfrak{X}, \eta}$.

Then
\footnotesize
\begin{equation*}
\operatorname{H}_{\et}^{n}
(
K, 
\mu_{p^{r}}^{\otimes(n-1)}
)
\to
\bigoplus_{\eta\in Y^{(0)}}
\operatorname{H}_{\et}^{n}
(
K_{(\eta)}, 
\mu_{p^{r}}^{\otimes(n-1)}
)
\oplus
\bigoplus_{x\in \mathfrak{X}^{(1)}\backslash Y^{(0)}}
\operatorname{H}_{\et}^{n-1}
(
\kappa(x), 
\mu_{p^{r}}^{\otimes(n-2)}
)
\end{equation*}
\normalsize
is injective for $n\geq 2$.
\end{thm}
\begin{proof}\upshape
The proof of the statement is same as
\cite[Theorem 2.5]{Hu}.
The statement follows from 
Proposition \ref{RHGL} and Lemma \ref{WA}.
\end{proof}

%



%

\end{document}